\documentclass{amsart}
\usepackage{enumerate}
\usepackage{amssymb}
\usepackage{amsfonts}
\usepackage{amsmath}
\usepackage{mathtools}
\usepackage{graphicx}
\usepackage{url}
\usepackage[mathscr]{euscript}
 \let\mathscr\relax
\usepackage[scr]{rsfso}
\usepackage{enumerate}
\usepackage{xcolor}
\usepackage{mathrsfs}
\usepackage{amscd}
\usepackage{url}

\newcommand{\ideal}[1]{\operatorname{Ideal}[#1]}
\newcommand{\idea}{ \operatorname{Ideal} }
\newcommand{\qmod}[1]{\operatorname{QM}[#1]}
\newcommand{\qm}{\operatorname{QM} }
\newcommand{\pre}[1]{\operatorname{Pre}[#1]}
\newcommand{\po}{\operatorname{Pre}}

\newcommand{\xdt}[1]{x_{\Delta_{#1}}}
\newcommand{\udt}[1]{u_{\Delta_{#1}}}

\newcommand{\re}{\mathbb{R}}

\newcommand{\N}{\mathbb{N}}

\newcommand{\U}{\mathbb{U}}

\newcommand{\diag}{\mbox{diag}}

\newcommand{\lmd}{\lambda}

\newcommand{\eps}{\epsilon}

\newcommand{\Dt}{\Delta}
\def\af{\alpha}

\def\gm{\gamma}
\def\rank{\mbox{rank}}

\newcommand{\sig}{\sigma}
\newcommand{\Sig}{\Sigma}

\newcommand{\st}{\mathrm{s.t.}}

\newcommand{\reff}[1]{(\ref{#1})}

\newcommand{\mc}[1]{\mathcal{#1}}

\newcommand{\ddd}{,\ldots,}
\newcommand{\lip}{\left<}
\newcommand{\rip}{\right>}

\newcommand{\bdes}{\begin{description}}
	\newcommand{\edes}{\end{description}}

\newcommand{\bal}{\begin{align}}
	\newcommand{\eal}{\end{align}}

\newcommand{\bnum}{\begin{enumerate}}
	\newcommand{\enum}{\end{enumerate}}

\newcommand{\bit}{\begin{itemize}}
	\newcommand{\eit}{\end{itemize}}

\newcommand{\bea}{\begin{eqnarray}}
	\newcommand{\eea}{\end{eqnarray}}
\newcommand{\be}{\begin{equation}}
	\newcommand{\ee}{\end{equation}}

\newcommand{\baray}{\begin{array}}
	\newcommand{\earay}{\end{array}}

\newcommand{\bsry}{\begin{subarray}}
	\newcommand{\esry}{\end{subarray}}

\newcommand{\bca}{\begin{cases}}
	\newcommand{\eca}{\end{cases}}

\newcommand{\bcen}{\begin{center}}
	\newcommand{\ecen}{\end{center}}

\newcommand{\bbm}{\begin{bmatrix}}
	\newcommand{\ebm}{\end{bmatrix}}

\newcommand{\bmx}{\begin{matrix}}
	\newcommand{\emx}{\end{matrix}}

\newcommand{\bpm}{\begin{pmatrix}}
	\newcommand{\epm}{\end{pmatrix}}

\newcommand{\btab}{\begin{tabular}}
	\newcommand{\etab}{\end{tabular}}

\theoremstyle{plain}
\newtheorem{theorem}{Theorem}[section]
\newtheorem{thm}[theorem]{Theorem}

\newtheorem{lem}[theorem]{Lemma}

\newtheorem{ass}[theorem]{Assumption}

\theoremstyle{definition}

\newtheorem{eg}[theorem]{Example}

\numberwithin{equation}{section}

\begin{document}

\title[On Tightness of the Sparse Moment-SOS Hierarchy]
{A Characterization for Tightness of the Sparse Moment-SOS Hierarchy}

\author[Jiawang~Nie]{Jiawang Nie}
\author[Zheng~Qu]{Zheng Qu}
\author[Xindong~Tang]{Xindong Tang}
\author[Linghao~Zhang]{Linghao Zhang}

\address{Jiawang Nie and Linghao Zhang,
Department of Mathematics, University of California San Diego,
9500 Gilman Drive, La Jolla, CA, USA, 92093.}
\email{njw@math.ucsd.edu,liz010@ucsd.edu}

\address{Zheng Qu, Department of Applied Mathematics,
The Hong Kong Polytechnic University, Hung Hom, Kowloon, Hong Kong.}
\email{quzheng.qu@polyu.edu.hk}

\address{Xindong Tang, Department of Mathematics,
Hong Kong Baptist University,
Kowloon Tong, Kowloon, Hong Kong.}
\email{xdtang@hkbu.edu.hk}

\date{}

\begin{abstract}
This paper studies the sparse Moment-SOS hierarchy of relaxations
for solving sparse polynomial optimization problems.
We show that this sparse hierarchy is tight if and only if
the objective can be written as a sum of sparse nonnegative polynomials,
each of which belongs to the sum of the ideal and quadratic module
generated by the corresponding sparse constraints.
Based on this characterization, we give several sufficient conditions
for the sparse Moment-SOS hierarchy to be tight.
In particular, we show that this sparse hierarchy is tight
under some assumptions such as convexity,
optimality conditions or finiteness of constraining sets.
\end{abstract}



\maketitle

\section{Introduction}
Let $x \coloneqq (x_1, \ldots, x_n)$ be an $n$-dimensional vector variable in $\re^n$.
Suppose $\Dt_1, \ldots, \Dt_m$ are subsets of $[n] \coloneqq \{1,\ldots, n\}$.
For $\Dt_i = \{j_1, \ldots, j_{n_i} \}$, denote the subvector
$\xdt{i}  \coloneqq  (x_{j_1}, \ldots, x_{j_{n_i}} ).$
We consider the sparse polynomial optimization problem
\be  \label{sparse:pop}
\left\{ \baray{cl}
 \displaystyle \min_{x\in\re^n} &  f(x)  \coloneqq   f_1( \xdt{1} )+\dots + f_m( \xdt{m} ) \\
 \st  &  h_i( \xdt{i} ) = 0, \, g_i( \xdt{i} ) \ge 0, \,  i=1,\ldots, m.\\
\earay \right.
\ee
In the above, each $f_i$ is a polynomial and $h_i, g_i$ are vectors of polynomials in $\xdt{i}$.
We remark that for each $i$, the dimensions for the polynomial vectors $h_i$ and $g_i$ are not necessarily equal. It is also possible that $h_i$ or $g_i$ does not appear for some $i$.
Throughout the paper, a minimizer for \reff{sparse:pop}
means it is a global minimizer.
We denote by $f_{\min}$ the minimum value of \reff{sparse:pop} and denote
\be\label{K_di}
K_{\Dt_i} \coloneqq \{\xdt{i} \in \re^{n_i}: h_i(\xdt{i}) = 0, g_i(\xdt{i}) \ge 0 \}.
\ee
The feasible set $K$ of (\ref{sparse:pop}) is
\[
K = \bigcap\limits_{i = 1 }^{m}\{x \in \re^n : \xdt{i} \in K_{\Dt_i}\}.
\]

General polynomial optimization problems can be solved by
the Moment-SOS hierarchy of semidefinite programming relaxations proposed by Lasserre \cite{Lasserre2001}.
It produces a sequence of lower bounds for $f_{\min}$,
which converges to $f_{\min}$ under the archimedeanness.
Moreover, under some classical sufficient optimality conditions,
the Moment-SOS hierarchy in \cite{Lasserre2001} is {\it tight},
i.e., it has finite convergence \cite{Nie14}.
Throughout the paper, the method in \cite{Lasserre2001}
is called the {\it dense} Moment-SOS hierarchy.
We refer to the books \cite{HKL20,lasserre2015introduction,Lau09,nie2023moment}
about this topic.

The Moment-SOS hierarchy has strong performance for solving polynomial optimization.
However, a concern in its computational practice is that the sizes
of the resulting semidefinite programs grow quickly as the relaxation order increases.
To improve computational efficiency, it is important to exploit sparsity.
In some literature, the sparsity pattern in (\ref{sparse:pop})
is called {\it correlative sparsity} \cite{Lasserre06,waki2006sums},
to be distinguished from {\it term sparsity} \cite{TSSOS,ChTS,CSTSSOS}.
Some sparsity patterns can be given by
arithmetic-geometric mean inequalities
\cite{IlWolf16,MCW2021Newton,MCW2021Sig}.
In this paper, we focus on correlative sparsity.

In this paper, we consider the {\it sparse Moment-SOS hierarchy of semidefinite relaxations}
for solving sparse polynomial optimization.
We refer to (\ref{eq:spa_sos})-(\ref{eq:spa_mom}) in Section~\ref{sc:char}
for the exact formulations of sparse Moment-SOS relaxations.
The sparse relaxation has positive semidefinite (psd) matrix constraints
whose sizes are much smaller than those of the dense relaxation.
The sets $\Delta_1, \ldots, \Delta_m$ are said to satisfy the
running intersection property (RIP) if for every $j = 2, \ldots, m$, it holds
\be\label{rip}
\Delta_{j}  \cap (
\Delta_1 \cup \cdots \cup  \Delta_{j-1} )\subseteq \Delta_t
\ee
for some $t \le j-1$. The geometric meaning of RIP is as follows. Let $G=(V,E)$ be the graph associated with the correlative sparsity pattern $(\Delta_1,\ldots,\Delta_m)$, i.e., $V=\{1,\ldots,n\}$ and $(k_1,k_2)\in E$ if and only if $\{k_1,k_2\}\subseteq \Delta_i$ for some $i$.  The RIP is equivalent to that the sparsity pattern graph $G$ is \textit{chordal}, i.e., all its cycles of length at least four have a chord (an edge is called a chord if it joins two non-adjacent nodes in the cycle). We refer to~\cite{Blair1993,ChTS} for more details on chordal graphs and the RIP.
Let $f_k^{spa}$ and $f_k^{smo}$ denote the optimal values of
\reff{eq:spa_sos} and \reff{eq:spa_mom} respectively.
When the RIP holds and every $K_{\Dt_i}$ satisfies the archimedean condition,
the sequence of $f_k^{spa}$ converges to the minimum value
$f_{\min}$ of \reff{sparse:pop} asymptotically \cite{Kojima09,Lasserre06}.
A convergence rate of the sparse Moment-SOS hierarchy
is given in the recent work \cite{korda2023convergence}.
We refer to \cite{huangkang,KLMS23,Lasserre06,Magron23book,nie2009sparse,Qu2024,waki2006sums}
for related work on the sparse polynomial optimization.
Recently, sparsity has also been exploited to solve
noncommutative polynomial optimization \cite{KMP22,WangMag21}.
Some applications can be found in
\cite{newton2023sparse,wang2022certifying,yang2020perfect,yangdu,zhou2023semidefinite}.
Sparse Moment-SOS relaxations can be implemented in the software
{\tt TSSOS} \cite{vmjw21,TSSOS,CSTSSOS}.

In practice, people often observe that the sparse Moment-SOS hierarchy
has finite convergence, i.e., it is tight.
However, there exist examples for which the dense hierarchy is tight
while the sparse one is not (see \cite{nie2009sparse}).
To the best of the authors' knowledge, there is very little work to characterize
when the sparse Moment-SOS hierarchy is tight.

\subsection*{Contributions}
This paper characterizes tightness of the sparse Moment-SOS hierarchy,
which is given in (\ref{eq:spa_sos})-(\ref{eq:spa_mom}).
Our main contributions are:
\begin{itemize}
\item We give a sufficient and necessary condition for the sparse Moment-SOS hierarchy to be tight.
More precisely, when the optimal value of \reff{eq:spa_sos} is achievable,
we show that $f_k^{spa} = f_{\min}$ if and only if there exist polynomials
$p_i \in \re[\xdt{i}]$ such that (see Section~\ref{sc:pre} for the meaning of notation below)
\be\label{p+f}
\boxed{
\begin{gathered}
 p_1 + \cdots + p_m + f_{\min} = 0,   \\
f_i + p_i \in \idea_{\Dt_i}[h_i]_{2k} + \qm_{\Dt_i}[g_i]_{2k}, \ i \in [m].
\end{gathered}
}
\ee
We remark that the first equation in \reff{p+f} is equivalent to
\[
 f - f_{\min}  \, = \,  (f_1 + p_1) + \cdots + (f_m + p_m) .
\]
	
\item Under certain conditions, we show that the tightness of sparse Moment-SOS relaxations can be certified when flat truncations hold for the minimizer of the sparse moment relaxation \reff{eq:spa_mom}.
This also gives minimizers for \reff{sparse:pop}.
	
\item For convex sparse polynomial optimization problems, we show that the sparse Moment-SOS hierarchy is tight under some general conditions. In particular, we show that the sparse relaxations are tight for all relaxation orders, when \reff{sparse:pop} is an SOS-convex optimization problem.
	
\item Based on the characterization, we prove the sparse Moment-SOS hierarchy is tight when some classical sufficient optimality conditions hold or when each individual equality constraining variety is finite. In particular, when the RIP holds, we show that this sparse hierarchy is tight if each individual constraining variety is finite.
	
\item Based on the characterization, we prove the Schm\"{u}dgen type sparse Moment-SOS hierarchy is tight under some assumptions. In particular, we show that this sparse hierarchy is tight when the RIP holds and each individual constraining set is finite.

\item We remark that this paper is the first work that characterizes tightness of
the sparse Moment-SOS hierarchy.
Also, we give a sparse version of the flat truncation condition to detect tightness and to extract minimizers for (\ref{sparse:pop}).
It is generally hard to check tightness of the sparse Moment-SOS hierarchy \cite{Vargas24}.
There are almost no such results in the prior existing work, except for the case that moment matrices corresponding to intersections of blocks are rank one \cite{Lasserre06}.
Moreover, we give several sufficient conditions for the sparse Moment-SOS hierarchy to be tight.
For the dense case, the analogues of these sufficient conditions are studied in \cite{deklerk2011,Las09,Nie13Finite,Nie14}.
However, the new sufficient conditions given in this paper are not studied in earlier work because the sparse Moment-SOS hierarchy requires additional assumptions for tightness.
In Section~\ref{sc:sufficient}, we show that these new sufficient conditions are satisfied for many cases; see, for instance, Theorems~\ref{tm:convex} and \ref{tm:ripfiniteass}.
We acknowledge that some of our results (e.g., Theorems~\ref{thm:sepa:notattain}, \ref{thm:spar2full:mom}, \ref{tm:finite_sosc}, \ref{tm:finite_variety}, \ref{thm:schm:finitevariety}) apply some techniques developed in the first author's earlier work \cite{nie2013certifying,Nie13Finite,Nie14,nie2023moment}.

\end{itemize}

This paper is organized as follows.
Some basics on polynomial optimization and algebraic geometry are reviewed in Section~\ref{sc:pre}.
In Section~\ref{sc:char}, we give a characterization for the tightness of sparse Moment-SOS hierarchy,
and we study how to certify the tightness and get minimizers.
Section~\ref{sc:sufficient} gives sufficient conditions for the sparse Moment-SOS hierarchy to be tight.
The tightness of Schm\"{u}dgen type sparse Moment-SOS relaxations is investigated in Section~\ref{sc:schmudgen}.
Some numerical experiments are presented in Section~\ref{sc:example}.
Proofs for some theorems in earlier sections are given in Section~\ref{sc:some_proofs}.
Section~\ref{sc:conclusion} draws conclusions and makes some discussions.

\section{Preliminaries}
\label{sc:pre}

{\bf Notation}
Denote by $\N$ the set of nonnegative integers and $\re$ the real field.
For a positive integer $k$, let $[k] \coloneqq \{1\ddd k\}$.
For a subset $\Dt_i \subseteq [n]$, denote by $\re^{\Dt_i}$
the space of real vectors $x_{\Dt_i}$.
Let $\mathfrak{p}_i$ denote the projection from $\re^n$ to $\re^{\Dt_i}$ such that
$\mathfrak{p}_i (x) = \xdt{i}$ for all $x\in \re^n $.
For $i,j \in [n]$, denote $\Dt_{ij} \coloneqq \Dt_{i} \cap \Dt_{j}$ and the projection $\mathfrak{p}_{ij} : \re^{\Dt_i} \rightarrow \re^{\Dt_{ij}}$ be
such that $\mathfrak{p}_{ij}(\xdt{i}) = \xdt{ij}$.
The ring of polynomials in $x = (x_1, \ldots, x_n)$ with real coefficients is denoted as  $\re[x]$.
For a scalar or vector $p$ of polynomials, $\deg(p)$ denotes the maximal degree of its terms.
For a power $\af = (\af_1, \ldots, \af_n)$,
denote the monomial $x^\af \coloneqq x_1^{\af_1} \cdots x_n^{\af_n}$.
For a degree $d$, the subset of polynomials in $\re[x]$
with degrees at most $d$ is denoted as $\re[x]_d$.
We define $\re[x_{\Dt_i}]$ and $\re[x_{\Dt_i}]_d$ similarly.
For $\phi\in\re[x]$ and subsets $A,B\subseteq \re[x]$, denote
\[
\phi  A\, \coloneqq  \, \{ \phi p: p\in A \}, \quad
A + B\, \coloneqq  \, \{ p_1+p_2: p_1\in A,\, p_2\in B \}.
\]
For $f \in \re[x]$, $\nabla f$ denotes its gradient
with respect to $x$,
and $\nabla_{\Dt_i} f$ denotes the gradient
with respect to $x_{\Dt_i}$.
The Hessian matrices $\nabla^2f$ and $\nabla^2_{\Dt_i} f$ are defined similarly.
The Euclidean norm of $x$ is
$\|x\| \coloneqq (|x_1|^2 + \cdots + |x_n|^2)^{1/2}$.
Denote by $e_i$ the canonical basis vector such that the $i$th entry is $1$ and $0$ otherwise.
Let $\mc{S}^n$ be the space of all $n \times n$ real symmetric matrices.
For $X \in \mc{S}^n$, $X \succeq 0$ (resp., $X \succ 0$)
means $X$ is positive semidefinite (resp., positive definite).

\subsection{Ideals and SOS polynomials}
\label{ssc:idqm}

A tuple $h_i \coloneqq (h_{i,1}, \ldots, h_{i,\ell_i})$
of polynomials in $\re[\xdt{i}]$ generates
two ideals respectively in $\re[x_{\Dt_i}]$ and $\re[x]$ as
\[
\baray{rcl}
\idea_{\Dt_i}[h_i] & \coloneqq & h_{i,1} \re[\xdt{i}] + \cdots + h_{i,\ell_i}  \re[\xdt{i}], \\
\ideal{h_i}  & \coloneqq & h_{i,1}  \re[x] + \cdots + h_{i,\ell_i}  \re[x].
\earay
\]
The degree-$2k$ truncation of $\idea_{\Dt_i}[h_i]$ is
\[
\idea_{\Dt_i}[h_i]_{2k} \coloneqq h_{i,1}  \re[\xdt{i}]_{2k-\deg(h_{i,1})} + \cdots + h_{i,\ell_i}  \re[\xdt{i}]_{2k-\deg(h_{i,\ell_i})}.
\]
The truncation $\ideal{h_i}_{2k}$ is similarly defined.
For $h \coloneqq (h_1\ddd h_m)$, with each $h_i$ a tuple of polynomials in $\re[\xdt{i}]$,
we denote
\be\label{eq:idl_spa}
\left\{
\baray{rcl}
\ideal{h}_{spa} & \coloneqq & \idea_{\Dt_1}[h_1] + \cdots + \idea_{\Dt_m}[h_m],\\
\ideal{h}_{spa,2k} & \coloneqq & \idea_{\Dt_1}[h_1]_{2k} + \cdots + \idea_{\Dt_m}[h_m]_{2k}.
\earay
\right.
\ee
For a set $P \subseteq \re[\xdt{i}]$, its {\it real variety} is
\[
V_{\re}(P) \coloneqq \{ \xdt{i} \in \re^{\Dt_i} : p(\xdt{i}) = 0\,\, \forall p\in P \}.
\]
The {\it vanishing ideal} of a subset $V_i \subseteq \re^{\Dt_{i}}$ is
\[
 I(V_i)\coloneqq \{ p \in \re[x_{\Dt_i}] : p(\xdt{i}) = 0\,\,
\forall \xdt{i}\in V_i \}.
\]
Clearly, $P \subseteq I(V_{\re}(P))$.
An ideal $J$ is said to be
{\it real radical} if $J = I(V_{\re}(J))$.

A polynomial $\sigma \in \re[x_{\Dt_i}]$ is said to be a {\it sum of squares} (SOS)
if there exist polynomials $p_1, \ldots, p_s \in \re[x_{\Dt_i}]$ such that
$\sigma = p_1^2 + \cdots + p_s^2$.
The cone of all SOS polynomials in $\re[\xdt{i}]$ is denoted as
$\Sigma[\xdt{i}]$, and its degree-$2k$ truncation is
\[
\Sigma[\xdt{i}]_{2k} \coloneqq  \Sigma[\xdt{i}] \cap \re[\xdt{i}]_{2k}.
\]
We define $\Sig[x]$ and $\Sig[x]_{2k}$ similarly.
A tuple $g_i \coloneqq (g_{i,1}, \ldots, g_{i,s_i})$
of polynomials in $\re[\xdt{i}]$ generates
the quadratic modules in $\re[\xdt{i}]$ and $\re[x]$  respectively as
\[
\baray{rcl}
\qm_{\Dt_i}[g_i] &\coloneqq&  \Sig[\xdt{i}] + g_{i,1} \Sig[\xdt{i}] + \cdots + g_{i,s_i} \Sig[\xdt{i}], \\
\qm[g_i] &\coloneqq&  \Sig[x] + g_{i,1} \Sig[x] + \cdots + g_{i,s_i} \Sig[x] .
\earay
\]
The preorderings of $g_i$ in $\re[\xdt{i}]$ and $\re[x]$ are respectively
\[
\po_{\Dt_i}[g_i] \, \coloneqq \sum_{J \subseteq [s_i] }   \prod_{j \in J} g_{i,j}   \Sig[\xdt{i}], \quad
\po[g_i] \, \coloneqq \sum_{J \subseteq [s_i] }   \prod_{j \in J} g_{i,j}   \Sig[x].
\]
The product in the above is $1$ if $J$ is empty.
For an even degree $2k$, we denote the truncation ($g_{i,0} = 1$)
\[
\qm_{\Dt_i}[g_i]_{2k} \coloneqq \Big\{ \sum_{j=0}^{s_i} \sig_jg_{i,j} :
\sig_j \in \Sig[\xdt{i}], \deg(\sig_jg_{i,j}) \le 2k \Big\}.
\]
The truncation $\po_{\Dt_i}[g_i]_{2k}$ is defined similarly.
For $g \coloneqq (g_1\ddd g_m)$ with each $g_i$
a tuple of polynomials in $\re[\xdt{i}]$, we denote
\be\label{eq:qm_spa}
\left\{
\begin{array}{rcl}
\qm[g]_{spa} & \coloneqq & \qm_{\Dt_1}[g_1] + \cdots + \qm_{\Dt_m}[g_m],\\
\po[g]_{spa} & \coloneqq & \po_{\Dt_1}[g_1] + \ldots + \po_{\Dt_m}[g_m],\\
\qm[g]_{spa,2k} & \coloneqq & \qm_{\Dt_1}[g_1]_{2k} + \cdots + \qm_{\Dt_m}[g_m]_{2k}, \\
\po[g]_{spa,2k} & \coloneqq & \po_{\Dt_1}[g_1]_{2k} + \ldots + \po_{\Dt_m}[g_m]_{2k}.
\end{array}
\right.
\ee

The set $\idea_{\Dt_i}{[h_i]} + \qm_{\Dt_i}{[g_i]}$ is said to be {\it archimedean}
if it contains a polynomial $q$ such that
the set of all points $\xdt{i}$ satisfying $q(\xdt{i})\ge0$ is compact.
When $\idea_{\Dt_i}{[h_i]} + \qm_{\Dt_i}{[g_i]}$ is archimedean,
the set $K_{\Dt_i}$ must be compact.
The converse is not necessarily true.
However, if $K_{\Dt_i}$ is compact, i.e., there exists $R>0$ such that
$\Vert \xdt{i}\Vert^2  \le R$ for all $\xdt{i}\in K_{\Dt_i}$, then
$\idea_{\Dt_i}{[h_i]} + \qm_{\Dt_i}[g_i, R- \| \xdt{i} \|^2]$ is archimedean.
The following lemma is useful.

\begin{lem}  \label{lm:sep}
Let $f = f_{1}+\dots+f_m$ be such that each $f_i\in\re[\xdt{i}]$ and let
$G = G_1 + \cdots + G_m$ be such that each
$G_i \subseteq \re[\xdt{i}]$.
Then, for a given $\gamma\in\re$, it holds $f-\gamma \in G$
if and only if there exist polynomials $p_i \in \re[\xdt{i}]$ such that
\be\label{eq:sep_pi}
\boxed{
\begin{gathered}
p_1  + \cdots + p_m  + \gamma  =0 , \\
f_i + p_i \in G_i,\, i = 1, \ldots, m.
\end{gathered}
}
\ee
\end{lem}
\begin{proof}
The ``if" implication is straightforward.
For the ``only if'' implication, suppose
$f-\gamma \in G,$
then there exist $s_i\in G_i$ satisfying
\be\label{eq:f-fmin=phi+s}
f - \gm = s_1+\cdots+s_m.
\ee
For each $i$, let
$p_i\,\coloneqq\, s_i-f_i.$
Since $f_i,s_i\in \re[\xdt{i}]$, we have $p_i \in \re[\xdt{i}]$ and
\[\begin{aligned}
p_1+\cdots + p_m & = s_1+\cdots+s_m - (f_1+\cdots+f_m) \\
& = s_1+\cdots+s_m - f = -\gm.
\end{aligned}  \]
Therefore, $p_1+\cdots +p_m+\gm = 0$ and $f_i+p_i = s_i\in G_i$ for all $i$.
So, \reff{eq:sep_pi} holds.
\end{proof}

The feasible set of the sparse optimization problem \reff{sparse:pop} is
\be \label{set:K}
K = \left\{x \in \re^n
\left| \baray{c}
h_i(\xdt{i}) = 0, \  g_i(\xdt{i}) \ge 0, \\ i = 1, \ldots, m
\earay \right.
\right\}.
\ee
In some applications, one is interested in a certificate for $K$ to be empty.
The Positivstellensatz (see \cite{BCR98}) for $K = \emptyset$ is
\be\label{positiv}
1 + \sig + \phi = 0,
\ee
where $\sig \in \pre{g} \coloneqq \sum_{i=1}^{m} \pre{g_i}$ and $\phi \in \ideal{h} \coloneqq \sum_{i=1}^{m} \ideal{h_i}$.
The Positivstellensatz \reff{positiv} is said to be {\it sparse} if
\reff{positiv} holds for $\sig \in \pre{g}_{spa}$ and $\phi \in \ideal{h}_{spa}$.

\begin{thm}\label{tm:sparse_empty}
Let $K$ be the set as in \reff{set:K}. 
Then a sparse version of Positivstellensatz \reff{positiv}
holds if and only if there exist polynomials
$p_i \in \re[\xdt{i}]$ such that
\be\label{sparse:positiv}
\boxed{
\baray{c}
p_1  + \cdots + p_m  =0 , \\
-1 + p_i \in \idea_{\Dt_i}[h_i] + \po_{\Dt_i}[g_i], \,\, i \in [m].
\earay
}
\ee
\end{thm}
\begin{proof}
The conclusion follows from Lemma~\ref{lm:sep} with $\gm = 0$,
$f_i = -1$ and $G_i = \idea_{\Dt_i}[h_i]+\po_{\Dt_i}[g_i]$.
\end{proof}

\subsection{Sparse moments}
\label{ssc:sparmom}

Denote the set of monomial powers in $\xdt{i}$ as
\[
\N^{\Dt_i}\coloneqq \{\af = (\alpha_1\ddd \alpha_{n})\in\N^n :
\alpha_j = 0\,\, \forall j\notin \Dt_{i}\}.
\]
For a degree $d$, denote
$
\N_{d}^{\Dt_i} \coloneqq \{\af\in \N^{\Dt_i}:   |\af| \le d \},
$
where $|\af| \coloneqq \af_1 + \cdots + \af_{n}$.
The vector of all monomials in $\xdt{i}$ in the graded lexicographic order
with degrees up to $d$ is denoted as $[\xdt{i}]_d$, i.e.,
\be \label{[xDti]d}
[\xdt{i}]_d \, = \,  \big( x^\af \big)_{ \af \in \N_{d}^{\Dt_i}   }.
\ee
Denote by $\re^{\N_d^{\Dt_i}}$
the space of all real vectors labeled by $\af \in \mathbb{N}_d^{\Dt_i}$.
A vector $y_{\Dt_i}$ in $\re^{\N_{d}^{\Dt_i}}$
is called a truncated multi-sequence (tms) of degree $d$.
For a given $y_{\Dt_i}\in \re^{\mathbb{N}_d^{\Dt_i}}$,
the {\it Riesz functional} generated by $y_{\Dt_i}$
is the linear functional $\mathscr{R}_{y_{\Dt_i}}$ acting on $\re[\xdt{i}]$ such that
\[
\mathscr{R}_{y_{\Dt_i}}(x^{\alpha}) \, = \,
(y_{\Dt_i} )_{\alpha} \quad \mbox{for each} \quad \alpha\in \mathbb{N}_d^{\Dt_i} .
\]
This induces the bilinear operation
\be \label{<p,yDt>}
\langle p, y_{\Dt_i} \rangle \coloneqq \mathscr{R}_{y_{\Dt_i}}(p).
\ee
The {\it localizing vector} and the {\it localizing matrix} of $p\in\re[\xdt{i}]$,
generated by $y_{\Dt_i}$, are respectively
\[
\baray{rcl}
\mathscr{V}_{p}^{\Dt_i,2k}[y_{\Dt_i}] &\coloneqq& \mathscr{R}_{y_{\Dt_i}}\left( p(\xdt{i})  [\xdt{i}]_{k_1} \right), \\
L_{p}^{\Dt_i,k}[y_{\Dt_i}] &\coloneqq& \mathscr{R}_{y_{\Dt_i}}\left( p(\xdt{i})  [\xdt{i}]_{k_2}[\xdt{i}]^T_{k_2} \right).
\earay
\]
In some literature, the localizing matrix $L_{p}^{\Dt_i,k}[y_{\Dt_i}]$ is also denoted as $M_{k-\lceil \deg(p)/2 \rceil}(py_{\Dt_i})$; see \cite{Lasserre06,lasserre2015introduction,Magron23book,CSTSSOS}.
In the above, the Riesz functional is applied entry-wise and
\[
k_1 \coloneqq 2k-\deg(p), \quad k_2 \coloneqq \lfloor k-\deg(p)/2 \rfloor.
\]
In particular, when $p=1$ is the constant polynomial, we get the {\it moment matrix}
\[
M_{\Dt_i}^{(k)}[y_{\Dt_i}] \coloneqq L_{1}^{\Dt_i,k}[y_{\Dt_i}].
\]
Recall that $h_i = (h_{i,1}, \ldots, h_{i,\ell_i})$ and $g_i \coloneqq (g_{i,1}, \ldots, g_{i,s_i})$ are tuples of polynomials in $\re[\xdt{i}]$. We denote
\[
\baray{rcl}
\mathscr{V}_{h_i}^{\Dt_i, 2k}[y_{\Dt_{i}}] &\coloneqq&
  (\mathscr{V}_{h_{i,1}}^{\Dt_i,2k}[y_{\Dt_i}], \ldots,
\mathscr{V}_{h_{i,\ell_i}}^{\Dt_i,2k}[y_{\Dt_i}]), \\
L_{g_i}^{\Dt_i, k}[y_{\Dt_{i}}] &\coloneqq& \diag
   \left( L_{g_{i,1}}^{\Dt_i,k}[y_{\Dt_i}] \ddd L_{g_{i,s_i}}^{\Dt_i,k}[y_{\Dt_i}] \right).
\earay
\]
For a given $k$, denote the monomial power set
\be\label{u}
\U_k \coloneqq \bigcup_{i \in [m]}  \mathbb{N}_{2k}^{\Dt_i}.
\ee
Let $\re^{\U_k}$ denote the space of real vectors
$
y  =  (y_{\af})_{ \af \in \U_k }.
$
For given $y\in\re^{\U_k}$, we denote the subvector
\be \label{yDti}
y_{\Dt_i} \, \coloneqq \, (y_{\af})_{ \af \in \mathbb{N}_{2k}^{\Dt_i} }.
\ee
For the objective $f$ in \reff{sparse:pop}, we denote
\be \label{<f,y>}
\langle f, y \rangle  \, \coloneqq \,
\langle f_1, y_{\Dt_1} \rangle + \cdots + \langle f_m, y_{\Dt_m} \rangle .
\ee
For a degree $t \le k$, the $2t$-degree truncation of $y_{\Dt_i}$ is the subvector
\be \label{yDti:2t}
y_{\Dt_i}|_{2t} \, \coloneqq \, (y_{\af})_{ \af \in \mathbb{N}_{2t}^{\Dt_i} }.
\ee
The sparse localizing vectors/matrices are denoted as
\be \label{loc:mat:vec}
\left\{
\baray{rcl}
\mathscr{V}_{h_i}^{spa, 2k}[y] & \coloneqq & \mathscr{V}_{h_i}^{\Dt_i,2k}[y_{\Dt_i}], \\
L_{g_i}^{spa, k}[y] & \coloneqq &  L_{g_i}^{\Dt_i,k}[y_{\Dt_i}], \\
M^{(k)}_{spa}[y] &\coloneqq&  \diag \left(M_{\Dt_1}^{(k)}[y_{\Dt_1}] \ddd
M_{\Dt_m}^{(k)}[y_{\Dt_m}] \right).
\earay
\right.
\ee
We refer to \cite{HKL20,lasserre2015introduction,Lau09,nie2023moment}
for more detailed introductions to polynomial optimization.

\section{The characterization for tightness}
\label{sc:char}

In this section, we characterize when the sparse Moment-SOS hierarchy is tight
for solving \reff{sparse:pop}.
Let
\be \label{deg:k0}
k_0 \coloneqq \max_{i\in[m]} \left\{ \lceil \deg(f)/2\rceil ,
\lceil \deg(g_i)/2\rceil ,  \lceil \deg(h_i)/2 \rceil  \right\} .
\ee
For a degree $k\ge k_0$, the $k$th order sparse SOS relaxation for \reff{sparse:pop} is
\be \label{eq:spa_sos}
\left\{ \baray{ccl}
f_k^{spa}  \coloneqq  &  \max &  \gamma  \\
& \st  &  f - \gamma  \in  \ideal{h}_{spa, 2k} + \qmod{g}_{spa, 2k}.
\earay \right.
\ee
Its dual optimization problem is the sparse moment relaxation
\be\label{eq:spa_mom}
\left\{ \baray{rcl}
f_k^{smo}  \coloneqq   & \min &  \langle f, y \rangle = \langle f_1, y_{\Dt_1} \rangle + \cdots + \langle f_m, y_{\Dt_m} \rangle \\
&  \st   &  \mathscr{V}_{h_i}^{spa, 2k}[y] = 0, L_{g_i}^{spa, k}[y] \succeq 0 \ ( i \in [m]), \\
&     &  M_{spa}^{(k)}[y]\succeq 0, y_0=1, \, y \in \re^{\U_k}.
\earay \right.
\ee
We refer to Subsections \ref{ssc:idqm} and \ref{ssc:sparmom}
for the above notation.
Recall that $f_{\min}$ denotes the minimum value of \reff{sparse:pop}.
When the RIP holds, if each $\idea_{\Dt_i}{[h_i]} + \qm_{\Dt_i}{[g_i]}$ is archimedean,
then $f_k^{spa} \rightarrow f_{\min}$ as $k \rightarrow \infty$,
as shown in \cite{Lasserre06}.
When $f_k^{spa} = f_{\min}$ for some $k$, the hierarchy of \reff{eq:spa_sos}
is said to be tight. Similarly, if $f_k^{smo} = f_{\min}$ for some $k$,
then the hierarchy of \reff{eq:spa_mom} is tight.
If they are both tight, the sparse Moment-SOS hierarchy of
\reff{eq:spa_sos}-\reff{eq:spa_mom} is said to be tight,
or to have finite convergence.

\subsection{Characterization for the tightness}

First, we give a sufficient and necessary condition for the sparse SOS hierarchy
of \reff{eq:spa_sos} to be tight.
We remark that the sparse Moment-SOS hierarchy of \reff{eq:spa_sos}-\reff{eq:spa_mom}
is tight if and only if the sparse SOS hierarchy of~\reff{eq:spa_sos} is tight.
The following is the main theorem.

\begin{thm} \label{thm:sepa:match}
For the sparse Moment-SOS hierarchy of
\reff{eq:spa_sos}-\reff{eq:spa_mom}, we have:
\begin{enumerate}[(i)]
\item For a relaxation order $k \ge k_0$, it holds
\be
\label{eq:f-fmin_in_IQ_2k}
f-f_{\min} \in \ideal{h}_{spa,2k} + \qmod{g}_{spa,2k}
\ee
if and only if there exist polynomials
$p_i \in \re[\xdt{i}]_{2k}$ such that
\be
\label{eq:p+fmin_in_idl}
\boxed{
\begin{gathered}
p_1  + \cdots + p_m  + f_{\min}  =0, \\
f_i + p_i \in \idea_{\Dt_i}[h_i]_{2k} + \qm_{\Dt_i}[g_i]_{2k}, \ i \in [m].
\end{gathered}
}
\ee
The equation in the above is equivalent to that
\[
f - f_{\min}  \, = \, (f_1 + p_1)  + \cdots + (f_m + p_m ).
\]

\item When \reff{eq:p+fmin_in_idl} holds for some order $k$,
the minimum value $f_{\min}$ of (\ref{sparse:pop})
is achievable if and only if all sparse polynomials $f_i + p_i$
have a common zero in $K$, i.e., there exists $u\in K$ such that
$f_i(\udt{i})+ p_i(\udt{i})=0$ for all $i\in[m]$.

\end{enumerate}
\end{thm}
\begin{proof}
(i) Let $\gm = f_{\min}$,
$G = \ideal{h}_{spa,2k} + \qm[g]_{spa, 2k}$ and
\[
G_i = \idea_{\Dt_i}[h_i]_{2k} +\qm_{\Dt_{i}}[g_i]_{2k}, \, \, i=1,\ldots, m.
\]
Then, the conclusion follows from Lemma~\ref{lm:sep}.

\smallskip
\noindent (ii) Assume \reff{eq:p+fmin_in_idl} holds for some order $k$. \\
($\Rightarrow$): Suppose $f_{\min}$ is achievable for \reff{sparse:pop},
then there exists a minimizer $u \in K$ such that $f_{\min} = f(u)$. Note
\begin{eqnarray*}
f(x)-f_{\min} &=& - \Big[ f_{\min}+\sum_{i=1}^m p_i(\xdt{i}) \Big] +
  \sum_{i=1}^m [f_i(\xdt{i}) + p_i(\xdt{i})] \\
& \in &  \ideal{h}_{spa,2k} + \qm[g]_{spa,2k}.
\end{eqnarray*}
Since $p_1 + \dots + p_m +f_{\min} = 0,$  we have
\[
\sum_{i=1}^m (f_i(\udt{i}) + p_i(\udt{i})) = 0.
\]
Since each $f_i(\udt{i}) + p_i(\udt{i}) \ge 0$ on $K_{\Dt_i}$,
we have
$
f_i(\udt{i}) + p_i(\udt{i}) = 0
$
for all $i \in [m]$. Therefore,
$u \in K$ is a common zero of all $f_i+p_i$.

\smallskip
\noindent
($\Leftarrow$): Suppose $u\in K$ is a common zero of all $f_i+p_i$, then
\[
\sum\limits_{i=1}^{m} f_i(\udt{i})+ p_i(\udt{i}) = f(u) + p_1(\udt{1}) + \cdots + p_m(\udt{m}) = 0.
\]
Since $p_1  + \cdots + p_m  + f_{\min} =0$,
$f(u) = f_{\min}$, so $f_{\min}$ is achievable.
\end{proof}

Theorem~\ref{thm:sepa:match} gives a sufficient and necessary condition
for the membership \reff{eq:f-fmin_in_IQ_2k},
which implies the tightness $f^{spa}_{k} = f_{\min}$.
When the optimal value of \reff{eq:spa_sos} is achievable,
$f^{spa}_k =f_{\min}$ if and only if \reff{eq:f-fmin_in_IQ_2k} holds.
When \reff{eq:spa_sos} does not achieve its optimal value,
we may not have (\ref{eq:f-fmin_in_IQ_2k}),
even if the relaxation \reff{eq:spa_sos} is tight.
For such cases, we have the following characterization
for the tightness $f^{spa}_k =f_{\min}$.

\begin{thm}  \label{thm:sepa:notattain}
The $k$th order sparse SOS relaxation \reff{eq:spa_sos}
is tight (i.e., $f_{\min} = f^{spa}_k$)
if and only if for every $\epsilon>0$, there exist
polynomials $p_i \in \re[\xdt{i}]_{2k}$ such that
\be  \label{eq:sep_pi_eps}
\boxed{
\begin{gathered}
p_1  + \cdots + p_m  + f_{\min}  = 0,\\
f_i + p_i +\epsilon \in \idea_{\Dt_i}[h_i]_{2k} + \qm_{\Dt_i}[g_i]_{2k},\ i \in [m].
\end{gathered}
}
\ee
\end{thm}
\begin{proof}
($\Leftarrow$): Suppose \reff{eq:sep_pi_eps} holds for some polynomials
$p_i \in  \re[\xdt{i}]_{2k}$, then
\begin{eqnarray*}
f - (f_{\min} - m\eps) &=& -\Big[ f_{\min} + \sum_{i=1}^{m}p_i \Big] +
    \sum_{i=1}^{m} (f_i+p_i+\eps) \\
& \in & \ideal{h}_{spa, 2k} + \qm[g]_{spa, 2k}.
\end{eqnarray*}
This means that $\gm = f_{\min} - m\eps$ is feasible for the $k$th order sparse SOS relaxation
\reff{eq:spa_sos}, so $f_{k}^{spa} \ge f_{\min}-m\eps$.
Since $\eps > 0$ is arbitrary, we get $f_{k}^{spa} \ge f_{\min}$.
On the other hand, we always have $f_{k}^{spa} \le f_{\min}$,
so $f_{k}^{spa} = f_{\min}$.

\medskip \noindent
($\Rightarrow$): For an arbitrary $\eps>0$, let
\[
f^{\epsilon}_i(\xdt{i}) \, \coloneqq \,  f_i(\xdt{i})+\eps.
\]
Suppose the $k$th order relaxation \reff{eq:spa_sos} is tight.
Then, every $\gm < f_{\min} = f_{k}^{spa}$
is feasible for \reff{eq:spa_sos}.
This is because for every $\eps > 0$, there exists a feasible $\hat{\gm}$
such that $\hat{\gm} > f_{\min} - m\eps$.
So, $f - \hat{\gm} \in \ideal{h}_{spa, 2k} + \qmod{g}_{spa, 2k}$ and
\[
f- (f_{\min} - m\eps) = f - \hat{\gm} + [\hat{\gm} - (f_{\min} - m\eps)]
\in \ideal{h}_{spa, 2k} + \qmod{g}_{spa, 2k} ,
\]
since $\hat{\gm} - (f_{\min} - m\eps) > 0$.
Thus, $\gm = f_{\min} - m\eps$ is feasible for \reff{eq:spa_sos}, and
\[
f(x) - (f_{\min} - m\epsilon) = \sum_{i=1}^mf^{\epsilon}_i-f_{\min}
\in \ideal{h}_{spa,2k} + \qmod{g}_{spa,2k}.
\]
We apply Lemma~\ref{lm:sep} with $\gm = f_{\min}$ and $f_i = f_i^{\eps}$,
then there exist polynomials $p_i  \in \re[\xdt{i}]_{2k}$
such that $p_1 + \cdots + p_m + f_{\min} =0$ and for all $i$,
\[
f_i^{\eps} + p_i = f_i + p_i + \eps \in
\idea_{\Dt_{i}}[h_i]_{2k} + \qm_{\Dt_{i}}[g_i]_{2k}  .
\]
So, \reff{eq:sep_pi_eps} holds.
\end{proof}

\subsection{Detecting tightness and extracting minimizers}

We now discuss how to detect tightness of sparse Moment-SOS relaxations
\reff{eq:spa_sos}-\reff{eq:spa_mom}.
The tightness of \reff{eq:spa_sos} can be detected
if there exists a feasible point $u$
such that $f_{k}^{spa} = f(u)$; similarly,
the tightness of \reff{eq:spa_mom}
can be detected by $f_{k}^{smo} = f(u)$.
This issue is related to how to extract a minimizer
for \reff{sparse:pop} from the sparse moment relaxations.

For each $i \in [m]$, denote the degrees
\be\label{degreedi}
\baray{rcl}
d_i &\coloneqq & \max\{\lceil \deg(h_i) / 2 \rceil, \lceil \deg(g_i) / 2 \rceil\} .
\earay
\ee
Suppose $y^*$ is a minimizer of \reff{eq:spa_mom}
for a relaxation order $k \ge k_0$.
To extract a minimizer, we typically need to consider a representing measure
for the subvector $y^*_{\Dt_i}$.
Suppose there exists a degree $t \in [k_0,k]$ such that
\be \label{decomp_i}
\boxed{
\begin{gathered}
{y^*_{\Dt_{i}} |_{2t} = \lambda_{i,1}[u^{(i,1)}]_{2t} + \cdots + \lambda_{i,r_i}[u^{(i,r_i)}]_{2t},} \\
\lambda_{i,1} > 0, \ldots, \lambda_{i,r_i} > 0, \quad
\lambda_{i,1} + \cdots + \lambda_{i,r_i} = 1
\end{gathered}
}
\ee
for some points $u^{(i,j)} \in K_{\Dt_i}$. We refer to
\reff{[xDti]d}, \reff{yDti} and \reff{yDti:2t}
for the above notation. Denote the set
\be\label{suppxdti}
\mathcal{X}_{\Dt_i} \coloneqq \{u^{(i,1)}, \ldots, u^{(i,r_i)} \}.
\ee

\begin{thm} \label{optmin}
Suppose $y^*$ is a minimizer of \reff{eq:spa_mom} and there exists $t\in[k_0,k]$ such that \reff{decomp_i} holds for all $i\in[m]$.
Let $x^*=(x_1^*,\,\cdots,\,x_n^*)$ be a point such that each
$x^*_{\Dt_i}\in \mathcal{X}_{\Dt_i}$.
If $f_t^{smo} =f_k^{smo}$, then $x^*$
is a minimizer of \reff{sparse:pop}.
\end{thm}

The proof of Theorem~\ref{optmin} is given in Subsection~\ref{pf_optmin}.
For the dense Moment-SOS hierarchy (i.e., $m=1$),
the tightness is guaranteed if there exists a degree $t\in [k_0,k]$ such that (\ref{decomp_i}) holds.
This is because for the dense case, we have (let $r\coloneqq r_i$, $u^{(j)} \coloneqq u^{(1,j)}$ and $\lmd^{(j)} \coloneqq \lmd^{(1,j)}$ for notational convenience)
\[ f(u^{(j)}) \ge f_{\min} \ge  \lip f,y^*|_{2k} \rip = \lip f,y^*|_{2t} \rip = \lmd_1 f(u^{(1)}) + \cdots + \lmd_r f(u^{(r)}) , \]
which implies
\[ f(u^{(1)}) = \cdots = f(u^{(r)}) = \lip f,y^*|_{2k} \rip = f_{\min}. \]
However, for the sparse Moment-SOS hierarchy of (\ref{eq:spa_sos})-(\ref{eq:spa_mom}), even if (\ref{decomp_i}) holds for all $i$, we may not have $f^{smo}_k = f_{\min}$
and there may not exist a point $x^*=(x_1^*,\,\cdots,\,x_n^*)$
such that every $x_i^* \in \mc{X}_{\Dt_i}$.
This is demonstrated as follows.

\begin{eg} \label{ex:extraction-not-minimizer}
Consider the following sparse polynomial optimization problem:
\be\label{eq:extraction-x-X}
\left\{
\baray{cl}
\min\limits_{x \in \re^3} & f(x) = x_1 +x_2 + x_3 \\
\st  &  (x_1 - 1)(x_1 - 2)(x_1+x_2-6)= 0,\ x_1 - x_2 = 0, \\
	 &  (x_2 - 1)(x_2 - 2)(x_2+x_3-6) = 0,\ (x_2 + x_3 - 3)(x_2-3) = 0, \\
	 &  (x_3 - 1)(x_3 - 2)(x_1+x_3-6) = 0,\ x_1 - x_3 = 0. \\
\earay
\right.
\ee
In the above, $f_1 = x_1, f_2 = x_2, f_3=x_3,$
$\Dt_1 = \{1,2\}$, $\Dt_2 = \{2,3\}$, $\Dt_3 = \{1,3\}$, and
$$
\begin{aligned}
& h_1 = \big((x_1 - 1)(x_1 - 2)(x_1+x_2-6), \, x_1 - x_2\big), \\
& h_2 = \big((x_2 - 1)(x_2 - 2)(x_2+x_3-6),\, (x_2 + x_3 - 3)(x_2-3)\big), \\
& h_3 = \big((x_3 - 1)(x_3 - 2)(x_1+x_3-6),\, x_1 - x_3\big).
\end{aligned}
$$
Clearly, we can check that $K = \{(3,3,3)\}$ and
\begin{eqnarray}  \label{eq:uij}
K_{\Dt_1} = \Big\{\bbm 1 \\ 1 \ebm, \bbm 2 \\ 2 \ebm, \bbm 3 \\ 3 \ebm \Big\}, &&
K_{\Dt_2} = \Big\{\bbm 1 \\ 2 \ebm, \bbm 2 \\ 1 \ebm, \bbm 3 \\ 3 \ebm \Big\},  \\
K_{\Dt_3} =  \Big\{\bbm 1 \\ 1 \ebm, \bbm 2 \\ 2 \ebm, \bbm 3 \\ 3 \ebm \Big\}. &&  \nonumber
\end{eqnarray}
For each $k\ge 2$, let $y^*\in \re^{\mathbb{U}_k}$ be such that
\begin{eqnarray}  \label{eq:ystar}
y^*_{\Dt_1} = \frac{1}{2} \Big( \bbm 1 \\ 1 \ebm_{2k} + \bbm 2 \\ 2 \ebm_{2k} \Big), &&
y^*_{\Dt_2} = \frac{1}{2} \Big( \bbm 1 \\ 2 \ebm_{2k} + \bbm 2 \\ 1 \ebm_{2k} \Big), \\
y^*_{\Dt_3} = \frac{1}{2} \Big( \bbm 1 \\ 1 \ebm_{2k} + \bbm 2 \\ 2 \ebm_{2k} \Big). &&  \nonumber
\end{eqnarray}
The above $y^*$ is well-defined,
i.e., $(y_{\Dt_i}^*)_{\af} = (y_{\Dt_j}^*)_{\af}$ for all $\af \in \N_{2k}^{\Dt_i} \cap \N_{2k}^{\Dt_j}$,
and it is feasible for the sparse moment relaxtion (\ref{eq:spa_mom}).
So, $f^{smo}_k \le \langle f,y^* \rangle =\frac{9}{2}$.
Since
\begin{eqnarray*}
x_1+x_2+x_3 - \frac{9}{2}
&=& (x_2-2x_1+2)h_{1,2} - h_{2,2} + (2x_3-x_1-1)h_{3,2}  \\
&&  + \frac{3}{2}(x_1-x_2)^2 + \frac{1}{2}(x_2+x_3-3)^2 + \frac{3}{2}(x_1-x_3)^2 \\
& \in &  \ideal{h}_{spa,2} + \Sigma[\xdt{1}]_{2}+ \Sigma[\xdt{2}]_{2}+ \Sigma[\xdt{3}]_{2},
\end{eqnarray*}
we have $f^{smo}_k\ge f^{spa}_k \ge \frac{9}{2}$ for all $k\ge 2$.
Hence, $f^{smo}_k = f^{spa}_k = \frac{9}{2}$ and $y^*$ is a minimizer of (\ref{eq:spa_mom}).
However, $f^{smo}_k < f_{\min} = 9$,
since the only feasible point of~(\ref{eq:extraction-x-X}) is $(3,3,3)$.
This shows that condition (\ref{decomp_i})
is not sufficient for the sparse Moment-SOS hierachy to be tight.
\end{eg}

For the decomposition \reff{decomp_i} to hold,
we typically need to assume the {\it flat truncation} condition:
there exists $t \in [k_0,k]$ such that ($d_i$ is given in (\ref{degreedi}))
\be\label{FTspa}
\rank \, M_{\Dt_i}^{(t)}[y^*_{\Dt_i}] = \rank \, M_{\Dt_i}^{(t-d_i)}[y^*_{\Dt_i}].
\ee
When (\ref{FTspa}) holds, we have the decomposition (\ref{decomp_i}) with
$r_i  =  \rank \, M_{\Dt_i}^{(t)}[y^*_{\Dt_i}]$.
We refer to \cite{henrion2005detecting,Lau05} for this fact.
Under some genericity assumptions, the dense Moment-SOS hierarchy is tight if and only if
the flat truncation holds (see \cite[Theorem~2.2]{nie2013certifying}
or \cite[Section~5.3]{nie2023moment}).
For the sparse Moment-SOS hierarchy, if (\ref{FTspa}) holds with $t=k$ for all $i$ and $\rank \,  M_{\Dt_{ij}}^{(k)}[y^*_{\Dt_i}] = 1$ for all $i\ne j$ when $\Dt_{i}\cap \Dt_j\ne \emptyset$,
then $f^{smo}_k = f_{\min}$ and one can extract minimizers for (\ref{sparse:pop}); see \cite[Theorem~3.7]{Lasserre06}.
Moreover, by Theorem~\ref{optmin}, if (\ref{FTspa}) holds for some $t\le k$ for all $i$ and $f^{smo}_t = f^{smo}_k$, then each point $x^*$ such that every $\xdt{i}^*\in\mc{X}_{\Dt_i}$ minimizes (\ref{sparse:pop}), and tightness of the sparse Moment relaxation (\ref{eq:spa_mom}) is certified when such $x^*$ exists.

The connectedness of the sparsity pattern set $\mc{D} \coloneqq \{\Dt_1, \ldots, \Dt_m\}$
can be given by the graph $G = ([m], E)$ whose vertex set is $[m]$
and whose edge set $E$ consists of pairs $(i,j)$ such that $\Dt_i \cap \Dt_j \ne \emptyset.$
The set $\mc{D}$ is said to be a {\it connected cover} of $[n]$
if $G$ is a connected graph and $\Dt_1 \cup \cdots \cup \Dt_m = [n]$.
The following gives conditions for the existence of
a representing measure for a truncation of $y^*$.
For a degree $t \le k$, denote (refer to \reff{u} for the notation of $\U_t$)
\be \label{tms:spar:2t}
y^* |_{spa, 2t} = \big( y^*_{\af} \big)_{ \af \in \U_t}, \quad
[u]_{spa, 2t} =  \big( u^{\af} \big)_{ \af \in \U_t}.
\ee

\begin{thm}  \label{thm:spar2full:mom}
Suppose $\{\Dt_1, \ldots, \Dt_m\}$ is a connected cover of $[n]$ and satisfies the RIP.
Assume $y^*$ is a minimizer of \reff{eq:spa_mom} and \reff{FTspa} holds for each $i$.
Let $r_i$ be as in \reff{decomp_i}. Assume $r_1 = \cdots = r_m = r$.
If for all $i \ne j$ with $\Dt_{i} \cap \Dt_j \ne \emptyset$ it holds
\be\label{FTDt_ij}
r = \rank \, M_{\Dt_{ij} }^{(t)}[y^*_{\Dt_{ij}}] = \rank \, M_{\Dt_{ij}}^{(t-1)}[y^*_{\Dt_{ij}}],
\ee
then there are points $u^{(1)}, \ldots, u^{(r)} \in K$ such that
\be\label{decomp_spa_general}
\boxed{
\begin{gathered}
y^* |_{spa, 2t} = \lambda_1[u^{(1)}]_{spa, 2t} + \cdots + \lambda_r[u^{(r)}]_{spa, 2t}, \\
\lambda_1 >0, \ldots, \lambda_r >0, \, \lambda_1 + \cdots + \lambda_r = 1.
\end{gathered}
}
\ee
Moreover, $u^{(1)}, \ldots, u^{(r)}$ are minimizers of \reff{sparse:pop}.
\end{thm}

The proof of Theorem~\ref{thm:spar2full:mom}
is given in Subsection~\ref{pf_thm:spar2full:mom}.
For the dense case (i.e., $m=1$),
if the flat truncation (\ref{FTspa}) holds,
then there exists a unique decomposition (up to permutations) given as in (\ref{decomp_spa_general}),
and the tightness of the dense moment relaxation can be certified.
However, for the sparse case, the condition (\ref{FTspa})
may not guarantee existence of (\ref{decomp_spa_general})
or tightness of the sparse moment relaxation (\ref{eq:spa_mom}).
To see this, consider the problem (\ref{eq:extraction-x-X})
in Example~\ref{ex:extraction-not-minimizer} and the minimizer
$y^*\in\re^{\mathbb{U}_2}$ as in (\ref{eq:ystar}).
Both conditions (\ref{FTspa}) and (\ref{FTDt_ij})
hold for all $i=1,2,3$ with $t=2$ and $r=2$.
But the decomposition (\ref{decomp_spa_general}) does not hold and
$ f^{smo}_k < f_{\min}$ for all $k$.
The RIP fails for Example~\ref{ex:extraction-not-minimizer}.
When RIP holds, conditions (\ref{FTspa}) and (\ref{FTDt_ij})
imply the decomposition (\ref{decomp_spa_general}) and tightness of
the sparse moment relaxation (\ref{eq:spa_mom}).

If the set $\mc{D} = \{\Dt_1, \ldots, \Dt_m\}$ is not a connected cover of $[n]$,
then $\mc{D}$ can be expressed as a union of disjoint subsets, say,
$\mc{D} = \mc{D}_1 \cup \cdots \cup \mc{D}_{\ell}, $
such that each $\mc{D}_i$ is a connected cover for a subset of $[n]$.
Then, the sparse optimization problem \reff{sparse:pop}
can be split into a union of smaller sized problems that do not have mutually joint variables.

\section{Some sufficient conditions for tightness}
\label{sc:sufficient}

This section gives some sufficient conditions for
the sparse Moment-SOS hierarchy to be tight.
The sparse hierarchy of (\ref{eq:spa_sos})-(\ref{eq:spa_mom})
is tight if and only if \reff{eq:p+fmin_in_idl} or \reff{eq:sep_pi_eps} holds for some $k$.
This leads to the following assumption.

\begin{ass}\label{eq:f+p>=0}
There exist polynomials $p_i \in \re[\xdt{i}]$ satisfying
\be  \label{assump}
\boxed{
\begin{array}{c}
\ p_1  + \cdots + p_m  + f_{\min}  =0 , \\
f_i + p_i \ge 0~ \text{on} ~ K_{\Dt_i}, \,  i=1\ddd m.
\end{array}
}
\ee
\end{ass}

For the dense case (i.e., $m=1$), Assumption~\ref{eq:f+p>=0} holds automatically for $p_1 = -f_{\min}$
if $-\infty<f_{\min}<+\infty$ (note that $f_{\min}=+\infty$ if $K=\emptyset$).
However, it may not hold for the sparse case. We refer to Example~\ref{ex:no_pi} for such an exposition.
Besides that, we remark that \reff{assump} does not imply
$f_i + p_i \in \mbox{Ideal}_{\Dt_i}[h_i] + \qm_{\Dt_i}[g_i]$.
This section explores various conditions for Assumption~\ref{eq:f+p>=0} to hold
and for the sparse Moment-SOS hierarchy of
(\ref{eq:spa_sos})-(\ref{eq:spa_mom}) to be tight.

\subsection{The convex case}

We consider the case that (\ref{sparse:pop}) is a convex optimization problem.
For each $i=1\ddd m$, we write that
\be \label{higi:tuple}
h_i = (h_{i,1}\ddd h_{i,\ell_i}),\quad  g_i = (g_{i,1}\ddd g_{i,s_i}).
\ee
Assume each $h_i$ is linear (or it does not appear),
all $f_{i}$ and $-g_{i,j}$ are convex,
and $u$ is a minimizer of \reff{sparse:pop}.
Under certain constraint qualification like Slater's condition
(i.e., there is a feasible point $x$ for \reff{sparse:pop} such that all $g_i(x)>0$),
there exist Lagrange multipliers $\lmd_{i,j}$ and $\nu_{i,j}$ such that
($\perp$ means that the product is $0$)
\be\label{eq:KKTequation}
\left\{
\begin{gathered}
\displaystyle \sum_{i=1}^m \nabla f_i(\udt{i}) = \sum_{i=1}^m \Big[\sum_{j=1}^{s_i}\lmd_{i,j}\nabla g_{i,j}(\udt{i}) + \sum_{j=1}^{\ell_i}\nu_{i,j}\nabla h_{i,j}(\udt{i})\Big],\\
0\le \lmd_{i,j}\perp g_{i,j}(\udt{i}) \ge 0,\ i = 1, \ldots, m, \ j=1,\ldots, s_i.
\end{gathered}
\right.
\ee
This is called the first order optimality condition (FOOC).
By Theorem~\ref{thm:sepa:match}, if Assumption~\ref{eq:f+p>=0} is satisfied with polynomials $p_i\in\re[\xdt{i}]$ such that every $f_i+p_i$ is convex,
then the sparse Moment-SOS hierarchy is tight under certain conditions; see \cite[Chapter~7.2]{nie2023moment}.
Interestingly, when (\ref{sparse:pop}) has a minimizer satisfying the FOOC (\ref{eq:KKTequation}),
such polynomials $p_i\in\re[\xdt{i}]$ always exist.
This is shown as follows.

\begin{thm}\label{tm:convex}
Suppose that each $h_i$ is linear (or it does not appear),
all $f_{i}$ and $-g_{i,j}$ are convex,
and $u$ is a minimizer of (\ref{sparse:pop}).
If (\ref{eq:KKTequation}) holds,
then Assumption~\ref{eq:f+p>=0} holds.
Furthermore, if in addition, the sub-Hessian
\[
\nabla_{\Dt_i}^2 \bigg[f_i(\udt{i}) - \sum_{j=1}^{s_i}\lmd_{i,j}g_{i,j}(\udt{i}) \bigg]
\]
is positive definite and each $\idea_{\Dt_i}[h_i] + \qm_{\Dt_i}[g_i]$
is archimedean, then the sparse Moment-SOS hierarchy of
(\ref{eq:spa_sos})-(\ref{eq:spa_mom}) is tight.
\end{thm}

\begin{proof}
For each $i=1\ddd m$, let
\[
p_i(x) \coloneqq -(x-u)^T\big[\nabla f_i(\udt{i}) - 
\sum_{j=1}^{s_i}\lmd_{i,j}\nabla g_{i,j}(\udt{i}) - 
\sum_{j=1}^{\ell_i}\nu_{i,j}\nabla h_{i,j}(\udt{i})\big] - f_i(\udt{i}).
\]
Since $f_i$, $g_i$, $h_i$ only depend on $\xdt{i}$,
we have $p_i \in \re[\xdt{i}]$. Note
\[
f_1(\udt{1}) + \cdots + f_m(\udt{m}) = f_{\min}.
\]
The equation in (\ref{eq:KKTequation}) and the above choice of $p_i$ imply
\[
p_1 +\cdots+ p_m + f_{\min} = 0,
\]
\[
\nabla f_i(\udt{i})+ \nabla p_i(\udt{i}) =  
\sum_{j=1}^{s_i}\lmd_{i,j}\nabla g_{i,j}(\udt{i}) +
\sum_{j=1}^{\ell_i}\nu_{i,j}\nabla h_{i,j}(\udt{i}),
\]
so $\udt{i}$ satisfies the FOOC for
\be\label{eq:minf+p} \left\{
\begin{array}{cl}
\displaystyle\min_{\xdt{i}} & f_i(\xdt{i})+p_i(\xdt{i}) \\
\st & h_i(\xdt{i}) = 0, \, g_i(\xdt{i}) \ge 0.
\end{array}
\right.
\ee
Since $p_i(\xdt{i})$ is linear, $f_i(\xdt{i}) + p_i(\xdt{i})$
is also convex in $\xdt{i}$,
hence $\udt{i}$ is a minimizer of (\ref{eq:minf+p}).
Its minimum value is $0$, so $f_i+p_i\ge0 $ on $K_{\Dt_i}$.
Therefore, Assumption~\ref{eq:f+p>=0} holds.
The tightness of the hierarchy of (\ref{eq:spa_sos})-(\ref{eq:spa_mom})
follows from Theorem~\ref{thm:sepa:match} and \cite[Corollary~3.3]{deklerk2011}
(or \cite[Theorem~7.2.5]{nie2023moment}).
\end{proof}

Theorems~\ref{optmin} and \ref{thm:spar2full:mom}
can be applied to get minimizers of \reff{sparse:pop}.
Note that a convex optimization problem typically has a unique optimizer,
or otherwise it has infinitely many optimizers.
For this reason, when the conditions \reff{decomp_i}-\reff{suppxdti}
and/or \reff{FTDt_ij}-\reff{decomp_spa_general} hold,
we typically have $r_i=1$ for all $i$.
For such cases, if $y^*$ is a minimizer of \reff{eq:spa_mom},  the point
\[
x^* \, \coloneqq \, (y^*_{e_1}, \ldots, y^*_{e_n} )
\]
is a minimizer of \reff{sparse:pop}.

Next, we consider the special case that the defining polynomials of (\ref{sparse:pop}) are SOS-convex.
Recall that a polynomial $\phi(x)$ is SOS-convex
if its Hessian $\nabla^2 \phi(x) = H(x)^T H(x)$
for some matrix polynomial $H(x)$.
We refer to \cite[Chap.~7]{nie2023moment} for SOS-convex polynomials.
For the dense case (i.e., $m=1$), when $f$ and all $-g_{i,j}$
are SOS-convex and $h$ is linear (or it does not appear),
the dense Moment-SOS relaxations are tight for all $k\ge k_0$; see \cite{Las09} or \cite[Section~7.2]{nie2023moment}.
In the following, we show that the sparse Moment-SOS relaxations
\reff{eq:spa_sos}-\reff{eq:spa_mom} are also tight for all $k \ge k_0$
and the condition \reff{eq:p+fmin_in_idl} holds, under the SOS-convexity assumption.

\begin{thm}\label{tm:sosconvex}
Assume the feasible set $K\ne \emptyset$ and
the minimum value $f_{\min} > -\infty$.
Suppose each $h_i$ is linear or it does not appear, each $f_{i}$ and $-g_{i,j}$
are SOS-convex in $x_{\Dt_i}$.
Then, we have:
\begin{enumerate}[(i)]
\item For all $k \ge k_0$, $f^{smo}_k = f_{\min}$.
Moreover, if Slater's condition holds, then $f^{spa}_k = f_{\min}$ and
\be \label{finIQ:soscvx}
f - f_{\min} \in \ideal{h}_{spa,2k}+\qmod{g}_{spa,2k}.
\ee

\item For every minimizer $y^*$ of \reff{eq:spa_mom}, the point
$
x^* = (y^*_{e_1}, \ldots, y^*_{e_n} )
$
is a minimizer of \reff{sparse:pop}.

\end{enumerate}
\end{thm}
\begin{proof}
Suppose $y$ is feasible for the relaxation~\reff{eq:spa_mom}
and let $v = (y_{e_1}, \ldots, y_{e_n})$. Since they are SOS-convex,
by Jensen's inequality (see \cite{Las09} or \cite[Chap.~7]{nie2023moment}), we have
\[
f_i(v_{\Dt_i}) \le \langle f_i, y_{\Dt_i} \rangle, \quad
-g_{i,j}(v_{\Dt_i}) \le - \langle g_{i,j}, y_{\Dt_i} \rangle, \quad
i=1,\ldots, m.
\]
When $h_i$ is linear, $h_i(v_{\Dt_i}) = \langle h_i, y_{\Dt_i} \rangle$.
The feasibility constraint in \reff{eq:spa_mom} implies that
$g_{i,j}(v_{\Dt_i}) \ge 0$, so $v$ is a feasible point for \reff{sparse:pop}.
Also note that
\[
f(v) = \sum_{i=1}^m  f_i(v_{\Dt_i}) \leq
\sum_{i=1}^m \langle f_i, y_{\Dt_i} \rangle = \langle f, y \rangle .
\]
The above holds for all $y$ that is feasible for \reff{eq:spa_mom},
so $f_{\min}  \leq f^{smo}_k$. On the other hand, we always have
$f_{\min}  \ge  f^{smo}_k$, so $f^{smo}_k = f_{\min}$.
Therefore, if $y^*$ is a minimizer of \reff{eq:spa_mom},
then $x^*$ is a minimizer of \reff{sparse:pop}.

When the Slater's condition holds,
the moment relaxation \reff{eq:spa_mom} has strictly feasible points.
So, the strong duality holds between
\reff{eq:spa_sos} and \reff{eq:spa_mom} and
\reff{eq:spa_sos} achieves its optimal value.
Therefore, $f^{spa}_k = f^{smo}_k = f_{\min}$
and \reff{finIQ:soscvx} holds.
\end{proof}

\subsection{The case of sufficient optimality conditions}

Under the archimedeanness,
the dense Moment-SOS hierarchy in \cite{Lasserre2001} is tight when
linear independence constraint qualification condition (LICQC),
strict complementarity condition (SCC) and
second order sufficient condition (SOSC)
hold at every minimizer, as shown in \cite{Nie14}.
We refer to \cite[Section~5.1]{nie2023moment} for these conditions.
For sparse Moment-SOS relaxations, we have a similar conclusion
under Assumption~\ref{eq:f+p>=0}.

\begin{thm}
\label{tm:finite_sosc}
Suppose Assumption~\ref{eq:f+p>=0} holds and each
$\idea_{\Dt_i}[h_i]+\qm_{\Dt_i}[g_i]$ is archimedean in $\re[\xdt{i}]$.
Assume the LICQC, SCC and SOSC hold at every minimizer of the optimization problem
\be \label{eq:split_pop}
\left\{ \baray{rl}
\min &  f_i( \xdt{i} ) + p_i( \xdt{i} )  \\
\st &  h_i( \xdt{i} )  =  0, \,
    g_i( \xdt{i} ) \ge  0.
\earay \right.
\ee
Then, it holds that
\begin{enumerate}[(i)]
\item When $k$ is big enough, we have $f^{spa}_k = f^{smo}_k = f_{\min}$.
Moreover, if each $\idea_{\Dt_i}[h_i]$ is real radical,
then
\be\label{fminachieved:opcd}
f - f_{\min} \in \ideal{h}_{spa}+\qmod{g}_{spa}.
\ee

\item When $k$ is big enough, every minimizer $y^*$ of \reff{eq:spa_mom}
satisfies the flat truncation condition: there exists a degree
$t \in [k_0,k] $ such that for every $i$,
\be\label{FTspa4}
r_i = \rank \, M_{\Dt_i}^{(t)}[y^*_{\Dt_i}] = \rank \, M_{\Dt_i}^{(t-d_i)}[y^*_{\Dt_i}],
\ee
where $d_i$ is as in \reff{degreedi}.
Therefore, the decomposition \reff{decomp_i} holds and
each point $x^*$ as in Theorem~\ref{optmin}
is a minimizer of \reff{sparse:pop}.

\item
Suppose $\{\Dt_1, \ldots, \Dt_m\}$ is a connected cover of $[n]$
(see the definition before Theorem~\ref{thm:spar2full:mom}) and satisfies the RIP.
Assume \reff{FTspa4} holds for all $i$ and let
$\mathcal{X}_{\Dt_i}$ be the set as in \reff{suppxdti}.
For all $i \ne j$ with $\Dt_{i} \cap \Dt_j \ne \emptyset$, suppose the projection
\[
\mathfrak{p}_{ij}(\mathcal{X}_{\Dt_i}) = \{\mathfrak{p}_{ij}(u^{(i,1)}), \ldots,
\mathfrak{p}_{ij}(u^{(i,r_i)}) \}
\]
is a set of $r_i$ distinct points. Then,
when $t$ is big enough, we have for all $i,j$
\be\label{FTyij}
r \coloneqq r_1 = \cdots = r_m =
 \rank \, M_{\Dt_{ij} }^{(t)}[y^*_{\Dt_{ij}}] =
\rank \, M_{\Dt_{ij}}^{(t-1)}[y^*_{\Dt_{ij}}],
\ee
and there exist $u^{(1)}, \ldots, u^{(r)} \in K$ such that
(see \reff{tms:spar:2t} for the notation)
\be\label{decomp_4.10}
\boxed{
\begin{gathered}
y^* |_{spa, 2t} = \lambda_1[u^{(1)}]_{spa, 2t} + \cdots + \lambda_r[u^{(r)}]_{spa, 2t}, \\
\lambda_1 >0, \ldots, \lambda_r>0,   \lambda_1 + \cdots + \lambda_r  =  1.
\end{gathered}
}
\ee
Moreover, $u^{(1)}, \ldots, u^{(r)}$ are minimizers of \reff{sparse:pop}.
\end{enumerate}
\end{thm}

The proof of Theorem~\ref{tm:finite_sosc} is given in Subsection~\ref{pf_tm:finite_sosc}.
For the dense case (i.e., $m=1$), the flat truncation (\ref{FTspa4})
is sufficient for the moment relaxation to be tight.
However, for the sparse case, the condition (\ref{FTspa4}) alone cannot guarantee tightness,
because it may not imply \reff{decomp_4.10}.
Example~\ref{ex:extraction-not-minimizer} is an exposition for this.

\subsection{The case of finite sets}

We discuss the case that each real variety $V_{\re}(h_i)$ is finite. In the following,
we prove the sparse Moment-SOS hierarchy of (\ref{eq:spa_sos})-(\ref{eq:spa_mom})
is tight if the equality constraint of \reff{eq:split_pop}
defines a finite real variety for each $i$, under Assumption~\ref{eq:f+p>=0}.
For the dense case (i.e., $m=1$), similar results are shown in \cite{Nie13Finite}.

\begin{thm}
\label{tm:finite_variety}
Suppose Assumption~\ref{eq:f+p>=0} holds and each real variety $V_{\re}(h_i)$ is finite.
Then, it holds that
\begin{enumerate}[(i)]
\item  When $k$ is big enough, we have $f^{spa}_k = f^{smo}_k = f_{\min}$.
Moreover, if each $\idea_{\Dt_i}[h_i]$ is real radical, then
\be\label{fminachieved:finite_var}
f - f_{\min} \in \ideal{h}_{spa,2k}+\qmod{g}_{spa,2k}.
\ee

\item When $k$ is big enough, every minimizer $y^*$ of \reff{eq:spa_mom}
satisfies the flat truncation condition: there exists a degree
$t\in [k_0,k]$ such that for every $i$,
\be\label{FTspa_finite_var}
r_i = \rank \, M_{\Dt_i}^{(t)}[y^*_{\Dt_i}] = \rank \, M_{\Dt_i}^{(t-d_i)}[y^*_{\Dt_i}],
\ee
where $d_i$ is as in \reff{degreedi}.
Therefore, the decomposition \reff{decomp_i} holds and
each point $x^*$ as in Theorem~\ref{optmin}
is a minimizer of \reff{sparse:pop}.

\item
Suppose $\{\Dt_1, \ldots, \Dt_m\}$ is a connected cover of $[n]$
(see the definition before Theorem~\ref{thm:spar2full:mom}) and satisfies the RIP.
Assume \reff{FTspa_finite_var} holds for all $i$ and let
$\mathcal{X}_{\Dt_i}$ be the set as in \reff{suppxdti}.
For all $i \ne j$ with $\Dt_{i} \cap \Dt_j \ne \emptyset$, suppose the projection
\[
\mathfrak{p}_{ij}(\mathcal{X}_{\Dt_i}) = \{\mathfrak{p}_{ij}(u^{(i,1)}), \ldots,
\mathfrak{p}_{ij}(u^{(i,r_i)}) \}
\]
is a set of $r_i$ distinct points. Then, when $t$ is big enough, we have for all $i,j$,
\be  \label{FTyD_ij}
\baray{c}
r \coloneqq r_1 = \cdots = r_m =
  \rank \, M_{\Dt_{ij} }^{(t)}[y^*_{\Dt_{ij}}] =
\rank \, M_{\Dt_{ij}}^{(t-1)}[y^*_{\Dt_{ij}}],
\earay
\ee
and there exist $u^{(1)}, \ldots, u^{(r)} \in K$ such that
(see \reff{tms:spar:2t} for the notation)
\be
\boxed{
\begin{gathered}
y^* |_{spa, 2t} = \lambda_1[u^{(1)}]_{spa, 2t} + \cdots + \lambda_r[u^{(r)}]_{spa, 2t}, \\
\lambda_1 >0, \ldots, \lambda_r>0,   \lambda_1 + \cdots + \lambda_r  =  1.
\end{gathered}
}
\ee
Moreover, $u^{(1)}, \ldots, u^{(r)}$ are minimizers of \reff{sparse:pop}.
\end{enumerate}
\end{thm}

The proof of Theorem~\ref{tm:finite_variety} is given in Subsection~\ref{pf_tm:finite_variety}.
Unlike the dense case, if we only assume that every real variety $V_{\re}(h_i)$ is finite,
then the sparse Moment-SOS hierarchy of (\ref{eq:spa_sos})-(\ref{eq:spa_mom})
may not be tight (as shown in Example~\ref{ex:extraction-not-minimizer}).
The sparsity pattern there ($\Dt_1 = \{1,2\}$, $\Dt_2 = \{2,3\}$, $\Dt_3 = \{1,3\}$)
does not satisfy the RIP.
However, if in addition, we assume the RIP, then Assumption~\ref{eq:f+p>=0}
holds and the sparse Moment-SOS hierarchy (\ref{eq:spa_sos})-(\ref{eq:spa_mom}) is tight, as shown below.

For each $t = 2,\ldots,m$, denote the sets
\begin{eqnarray}
\label{hatDt:t-1}
\widehat{\Dt}_{t-1} &\coloneqq&  \Dt_1 \cup \cdots \cup \Dt_{t-1}, \quad
I_t\coloneqq \widehat{\Dt}_{t-1} \cap \Delta_t, \\
\label{KhatDt:t-1}
K_{\widehat{\Dt}_{t-1}} &\coloneqq& \{
x_{\widehat{\Dt}_{t-1} }:  \xdt{i} \in K_{\Dt_i},
\enspace i=1,\ldots, t-1 \}.
\end{eqnarray}
For a set $S\subseteq \re^{\Gamma}$ with $\Gamma\supseteq I$, we denote the projection
\[ S_{I} \,\coloneqq\, \{ x_{I}: x_{\Gamma}\in S \}. \]

\begin{thm}\label{tm:ripfiniteass}
Assume that $K \ne \emptyset$, all $K_{\Delta_{i}}$ are compact, and
$\Delta_1,\ldots,\Delta_m$ satisfy the RIP. If the projections
$(K_{\widehat{\Dt}_{t-1}})_{I_t}$
and $(K_{\Delta_t})_{I_t}$
are both finite sets for all $2\le t\le m$,
then Assumption~\ref{eq:f+p>=0} holds.
\end{thm}
The proof of Theorem~\ref{tm:ripfiniteass} is given in Subsection~\ref{sc:proof_tm46}.
The following result is implied by Theorems~\ref{tm:finite_variety} and \ref{tm:ripfiniteass}.

\begin{thm}
Suppose that $K\ne\emptyset$, each real variety $V_{\re}(h_i)$ is a finite set, and
$\Delta_1,\ldots,\Delta_m$ satisfy the RIP.
Then, Assumption~\ref{eq:f+p>=0} holds and all
conclusions of Theorem~\ref{tm:finite_variety} hold.
\end{thm}

\section{The Schm\"{u}dgen type sparse relaxations}
\label{sc:schmudgen}

We write the tuples $g_i,h_i$ as in \reff{higi:tuple}.
The Schm\"{u}dgen type sparse SOS relaxation for solving \reff{sparse:pop} is
\be\label{eq:sch_sos}
\left\{ \baray{rcl}
f_{k}^{smg} \coloneqq & \max &  \gamma  \\
& \st  &  f- \gamma  \in  \ideal{h}_{spa, 2k} + \pre{g}_{spa, 2k}.
\earay \right.
\ee
The dual of \reff{eq:sch_sos}
is the Schm\"{u}dgen type sparse moment relaxation
\be\label{eq:sch_mom}
\left\{ \baray{ccl}
f_k^{smm}  \coloneqq   & \min &  \langle f, y \rangle = \langle f_1, y_{\Dt_1} \rangle + \cdots + \langle f_m, y_{\Dt_m} \rangle  \\
&  \st   &  \mathscr{V}_{h_i}^{spa, 2k}[y] = 0 \  (i \in [m]), \\
&     &  L_{g_{i,J}}^{spa, k}[y] \succeq 0 \ (i\in[m],\,J \subseteq [s_i] ), \\
&     &  y_0=1, \ y \in \re^{\U_k}.
\earay \right.
\ee
We refer to \reff{eq:qm_spa} and \reff{loc:mat:vec}
for the above notation. For a subset $J\subseteq [s_i]$, denote
$
g_{i, J} \coloneqq \prod_{j \in J} g_{i,j}.
$
For the case $J = \emptyset$, $g_{i,\emptyset} = 1$
and $L_{g_{i,\emptyset}}^{spa, k}[y]$ becomes the sparse moment matrix
$M_{\Dt_{i}}^{(k)}[y_{\Dt_{i}}]$.
Similar to Theorem~\ref{thm:sepa:match},
the theorem below follows from Lemma~\ref{lm:sep}.

\begin{thm}
\label{thm:schm:match}
For the Schm\"{u}dgen type relaxations \reff{eq:sch_sos}-\reff{eq:sch_mom}, we have:
\begin{enumerate}[(i)]
\item For each $k \ge k_0$, it holds
\be \label{f-fm:IQ:Schm}
f-f_{\min} \in \ideal{h}_{spa,2k} + \pre{g}_{spa,2k}
\ee
if and only if there exist polynomials
$p_i  \in \re[\xdt{i}]$ such that
\be
\label{eq:sep_pi_sch}
\boxed{
\begin{array}{c}
\ p_1  + \cdots + p_m  + f_{\min} =0   , \\
f_i + p_i \in \idea_{\Dt_i}[h_i]_{ 2k} +\po_{\Dt_i}[g_i]_{ 2k}, \, i \in [m].
\end{array}
}
\ee

\item When \reff{eq:sep_pi_sch} holds for some $k$,
the minimum value $f_{\min}$ of (\ref{sparse:pop}) is achievable
if and only if there is a common zero point in $K$ for all $f_i + p_i$.
\end{enumerate}
\end{thm}

For the dense case (i.e., $m=1$), if the feasible set is finite,
the Schm\"{u}dgen type dense Moment-SOS hierarchy is tight; see \cite[Theorem~4.1]{Nie13Finite}.
Interestingly, this result can be extended to the sparse case.

\begin{thm}
\label{thm:schm:finitevariety}
Suppose Assumption~\ref{eq:f+p>=0} holds and each $K_{\Dt_i}$ is a finite set. Then, we have:
\begin{enumerate}[(i)]
\item The Schm\"{u}dgen type sparse Moment-SOS hierarchy of
\reff{eq:sch_sos}-\reff{eq:sch_mom} is tight, i.e.,
$f_{k}^{smg} = f_{\min}$ when $k$ is big enough.

\item When $k$ is large enough, every minimizer $y^*$ of \reff{eq:sch_mom}
satisfies the flat truncation: there exists a degree $t\in [k_0,k]$ such that for every $i$,
\be \label{FT_smg}
r_i = \rank \, M_{\Dt_i}^{(t)}[y^*_{\Dt_i}] = \rank \, M_{\Dt_i}^{(t-d_i)}[y^*_{\Dt_i}],
\ee
where $d_i$ is as in \reff{degreedi}.
Therefore, the decomposition \reff{decomp_i} holds and
each point $x^*$ as in Theorem~\ref{optmin}
is a minimizer of \reff{sparse:pop}.

\item
Suppose $\{\Dt_1, \ldots, \Dt_m\}$ is a connected cover of $[n]$
(see the definition before Theorem~\ref{thm:spar2full:mom}) and satisfies the RIP.
Assume \reff{FT_smg} holds for all $i$.
Let $\mathcal{X}_{\Dt_i}$ be the set as in \reff{suppxdti}.
For all $i \ne j$ with $\Dt_{i} \cap \Dt_j \ne \emptyset$, suppose the projection
\[
\mathfrak{p}_{ij}(\mathcal{X}_{\Dt_i}) = \{\mathfrak{p}_{ij}(u^{(i,1)}), \ldots,
\mathfrak{p}_{ij}(u^{(i,r_i)}) \}
\]
is a set of $r_i$ distinct points. Then, when $t$ is big enough, we have for all $i,j$,
\be\label{FTyDt_ij}
r \coloneqq r_1 = \cdots = r_m =
\rank \, M_{\Dt_{ij} }^{(t)}[y^*_{\Dt_{ij}}] =
\rank  \, M_{\Dt_{ij}}^{(t-1)}[y^*_{\Dt_{ij}}],
\ee
and there exist $u^{(1)}, \ldots, u^{(r)} \in K$ such that (see \reff{tms:spar:2t} for the notation)
\be\label{dec_5.7}
\boxed{
\begin{gathered}
y^* |_{spa, 2t} = \lambda_1[u^{(1)}]_{spa, 2t} + \cdots + \lambda_r[u^{(r)}]_{spa, 2t}, \\
\lambda_1 >0, \ldots, \lambda_r>0, \quad   \lambda_1 + \cdots + \lambda_r  =  1.
\end{gathered}
}
\ee
Moreover, $u^{(1)}, \ldots, u^{(r)}$ are minimizers of \reff{sparse:pop}.

\end{enumerate}
\end{thm}

The proof of Theorem~\ref{thm:schm:finitevariety} is given in Subsection~\ref{pf_thm:schm:finitevariety}.
The flat truncation (\ref{FT_smg}) is sufficient for the dense Schm\"{u}dgen type moment relaxation (i.e., $m=1$) to be tight,
while this condition (\ref{FT_smg}) alone cannot guarantee tightness for the sparse case,
because it may not imply \reff{dec_5.7}.

To conclude this section, we show that the Schm\"{u}dgen type sparse Moment-SOS hierarchy of
(\ref{eq:sch_sos})-(\ref{eq:sch_mom}) is tight when all $K_{\Dt_i}$ are finite sets and the RIP holds.
This is an interesting generalization of Theorem~4.1 in \cite{Nie13Finite}.

\begin{thm}\label{cor:ripass}
Suppose all $K_{\Delta_{i}}$ are finite sets and $K\ne\emptyset$.
If $\Delta_1,\ldots,\Delta_m$ satisfy the RIP, then
Assumption~\ref{eq:f+p>=0} holds and all conclusions
in Theorem~\ref{thm:schm:finitevariety} hold.
\end{thm}
\begin{proof}
This follows directly from Theorem~\ref{tm:ripfiniteass}
and Theorem~\ref{thm:schm:finitevariety}.
\end{proof}

\section{Some examples}
\label{sc:example}

This section provides numerical experiments for
the sparse Moment-SOS hierarchy of (\ref{eq:spa_sos})-(\ref{eq:spa_mom}).
For Examples~\ref{ex:by_tm_33}-\ref{ex:4minimizers},
we use {\tt Yalmip} \cite{yalmip}
to implement sparse Moment-SOS relaxations, and apply {\tt Gloptipoly 3} \cite{GloPol3}
to check flat truncation conditions and extract minimizers.
For sparse convex optimization problems in Examples~\ref{ex:sosconv_exp}-\ref{ex:quartic},
we apply the software {\tt TSSOS} \cite{vmjw21,TSSOS,CSTSSOS}.
All semidefinite programs are solved by the software {\tt Mosek} \cite{mosek}.
The computation is implemented in Julia 1.10.3/MATLAB 2023b,
in an Apple MacBook Pro Laptop in macOS 14.2.1 with 12$\times$Apple M3 Pro CPU and RAM 18GB.
For neatness, only four decimal digits are displayed for computational results.

\begin{eg}
\label{ex:by_tm_33}
Consider the sparse optimization problem
\be
\left\{
\baray{cl}
\min\limits_{x \in \re^3} & f(x) =
\underbrace{x_1^2x_2(x_1^2+x_2-1)}_{f_1} +
\underbrace{x_2^2x_3(x_2^2+x_3-1)}_{f_2} \\
\st & 1-x_1^2-x_2^2 \ge 0, \quad 1- x_2^2 - x_3^2 \ge 0,
\earay
\right.
\ee
with $\Dt_{1} = \{1,2\}$ and $\Dt_2 = \{2, 3\}$.
By solving (\ref{eq:spa_sos})-(\ref{eq:spa_mom}) with $k=3$, we get $f^{spa}_3\approx 0.0666$.
The condition (\ref{eq:p+fmin_in_idl}) is satisfied for
\begin{eqnarray*}
p_1(\xdt{1}) &\approx& 0.04149+0.0426x_2-0.1275x_2^2-0.1107x_2^3+0.2197x_2^4 \\
&& +0.0669x_2^5-0.1037x_2^6,
\end{eqnarray*}
and
$p_2 = -f_{\min}-p_1$.
For $t=k=3$, (\ref{decomp_i}) holds for both $\Dt_1$ and $\Dt_2$, and
\[ \mc{X}_{\Dt_1} = \{ (-0.5100, 0.4798), (0.5100, 0.4798)\},\quad
 \mc{X}_{\Dt_2} = \{ (0.4798,0.3849)\}. \]
By Theorem~\ref{optmin}, we get two minimizers:
$
(\pm 0.5100, 0.4798, 0.3849).
$
\end{eg}

\begin{eg}
\label{ex:by_tm_34}
Consider the sparse optimization problem
\be
\left\{
\baray{cl}
\min\limits_{x \in \re^3} & f(x) =
\underbrace{x_1^2+4x_1x_2 }_{f_1} +
\underbrace{4x_2x_3 - x_3^2}_{f_2} \\
\st & 1-x_1^2-x_2^2 \ge 0, \quad 1- x_2^2 - x_3^2 \ge 0,
\earay
\right.
\ee
with $\Dt_{1} = \{1,2\}$ and $\Dt_2 = \{2, 3\}$.
Then, (\ref{eq:p+fmin_in_idl}) is satisfied with
\[
p_1(\xdt{1}) = 1.0000x_2^2+1.0000,\quad
p_2(\xdt{2}) = -1.0000x_2^2+3.0000.
\]
We solve (\ref{eq:spa_sos})-(\ref{eq:spa_mom}) with the order $k=2$.
For the minimizer $y^*$ of (\ref{eq:spa_mom}), the flat truncation condition (\ref{FTspa})
is satisfied for both $\Dt_1$ and $\Dt_2$ with $r_1 = r_2 =2$, and (\ref{FTDt_ij}) is satisfied with
\[
\rank \, M_{\Dt_{12} }^{(1)}[y^*_{\Dt_{12}}] =\rank \, M_{\Dt_{12} }^{(2)}[y^*_{\Dt_{12}}] = 2.\]
By Theorem~\ref{thm:spar2full:mom}, we get $f_{\min} = -4.0000$ and two minimizers:
\[
\pm(0.7071,-0.7071,0.7071).
\]
\end{eg}

\begin{eg}
\label{ex:4minimizers}
Consider the sparse optimization problem
\be
\left\{
\baray{cl}
\min\limits_{x \in \re^3} & f(x) = \underbrace{x_1x_2x_3-x_1x_2}_{f_1} +\underbrace{x_2x_3x_4-x_3x_4}_{f_2}   \\
\st  &  x_1^2 = x_1, \, x_2^2 = x_2, \, x_3^2 = x_3, \, x_4^2=x_4, \\
& x_1 + x_2 + x_3 \ge 1, x_2 + x_3 + x_4 \ge 1,
\earay
\right.
\ee
with $\Dt_{1} = \{1, 2, 3\}$, $\Dt_{2} = \{2, 3, 4\}$, and
\[
\begin{gathered}
h_1 = (x_1^2 - x_1, x_2^2 - x_2,x_3^2 - x_3),\quad g_1 = (x_1+x_2+x_3-1), \\
h_2 = (x_2^2 - x_2,x_3^2 - x_3,x_4^2 - x_4),\quad g_2 = (x_2+x_3+x_4-1).  \\
\end{gathered}
\]
We solve the sparse relaxation (\ref{eq:spa_sos})-(\ref{eq:spa_mom}) with $k=2$.
For the minimizer $y^*$ of (\ref{eq:spa_mom}), the condition (\ref{decomp_i})
holds with $t=k=2$ for both $\Dt_1$ and $\Dt_2$.
By Theorem~\ref{optmin}, we get $f_{\min} = f^{smo}_2=-1.0000$ and four minimizers:
\[ \begin{gathered}
(0.0000,0.0000,1.0000,1.0000), \quad
(1.0000,0.0000,1.0000,1.0000), \\
(1.0000,1.0000,0.0000,0.0000), \quad
(1.0000,1.0000,0.0000,1.0000). \\
\end{gathered}\]
Moreover, the condition (\ref{eq:p+fmin_in_idl}) is satisfied for
\begin{eqnarray*}
p_1(\xdt{1}) &\approx & 0.5001+0.4195x_2-0.2165x_3 -0.1484x_2x_3+0.1559x_2^2\\
&&-0.4111x_3^2-0.8565x_2^3+0.0166x_2^2x_3+0.1820x_2x_3^2-0.7216x_3^3\\
&&+0.6349x_2^2x_3^2-0.4251x_2^3x_3-0.2599x_2x_3^3+0.7810x_2^4+0.8491x_3^4,
\end{eqnarray*}
and $p_2 = f_{\min} - p_1$.
\end{eg}

\begin{eg}
\label{ex:sosconv_exp}
Consider the optimization problem:
\be
\left\{
\baray{cl}
\min\limits_{x \in \re^3} & f(x) = f_1(x_1,x_2) + f_2(x_2,x_3) + f_3(x_1,x_3)\\
\st & 1-x_1^4-x_2^4 \ge 0,\
  1- x_2^4-x_3^4 \ge 0,\  1-x^4_1-x_3^4 \ge 0.
\earay
\right.
\ee
In the above, $f_1\,\coloneqq\,x_1^6+x_2^6+x_1^3x_2^3+x_1$,
\[
\begin{gathered}
f_2\,\coloneqq\,x_2^6+x_3^6+x_2^3x_3^3-x_2,\quad
f_3\,\coloneqq\,x_1^6+x_3^6+x_1^3x_3^3+2x_3.
\end{gathered}
\]
The sparse relaxation (\ref{eq:spa_sos})-(\ref{eq:spa_mom})
is tight for all $k\ge3$. Solving it with $k=3$, we get
$
f^{spa}_{3}\approx -2.2561.
$
We get $f_{\min} = f^{spa}_{3}$ and the minimizer
$(-0.6036,0.6852,-0.7092).$
\end{eg}

\begin{eg}
\label{ex:qcqp}
Consider the convex quadratic optimization problem:
\be  \label{sparse:QCQP}
\left\{ \baray{cll}
  \displaystyle \min_{x\in\re^n} &  \displaystyle  \sum _{i=1}^m \big[ \xdt{i} ^T Q_i \xdt{i}+b_i ^T \xdt{i}\big] \\
 \st  &  1- \xdt{i}^T B_i \xdt{i}-c_i^T \xdt{i}  \geq 0 ,\  i=1,\ldots, m.
\earay \right.
\ee
We set $m=n$ and each block $\Delta_i$ consists of $w$ elements as
\be\label{eq:qcqp_delta_i}
\Dt_i = \left\{ \begin{array}{ll}
\{ i,\ldots,i+w-1 \} & \mbox{for } i\le n-w+1,\\
\{ i,\ldots,n,1,\ldots, w-n+i-1 \} & \mbox{for } i> n-w+1.
\end{array}
\right.
\ee
The $b_i$ and $c_i$ are randomly generated vectors obeying Gaussian distribution.
Each $Q_i = R_i^TR_i$ is randomly generated, where the entries of $R_i$ obey Gaussian distribution.
The matrix $B_i$ is generated in the same way.
By Theorem~\ref{tm:sosconvex}, the sparse relaxation (\ref{eq:spa_sos})-(\ref{eq:spa_mom}) is tight for all $k\ge1$.
We report the computational time (in seconds) for solving (\ref{sparse:QCQP}) by the sparse relaxation (\ref{eq:spa_sos})-(\ref{eq:spa_mom}) and dense relaxation in \cite{Lasserre2001} with order $k=1$.
In Table~\ref{tab:qcqp}, we display the time by the sparse relaxation (\ref{eq:spa_sos})-(\ref{eq:spa_mom}) on the left and the time by the dense relaxation on the right, inside each parenthesis.
For instance, when $n=100$ and $w=5$, the sparse relaxation took $0.04$ second,
while the dense relaxation took $2.24$ seconds.
When $n=500$, the dense relaxations
cannot be solved since the computer is out of memory (oom).

\begin{table}[htb]
\caption{The computational time (in seconds) for solving (\ref{sparse:QCQP})
by the sparse Moment-SOS relaxations (left) and the dense Moment-SOS relaxation (right).}
\label{tab:qcqp}
\centering
\begin{tabular}{|l|c|c|c|c|c}\hline
& $n=100$ & $n=200$ & $n=300$ & $n=500$  \\
\hline
 $w=5$ &  (0.04, 2.24) & (0.10, 82.01) & (0.17, 2205.36) & (0.35, oom)  \\
\hline
$w=10$ &(0.11, 2.58) &  (0.28, 87.02)& (0.44, 2221.67)& (0.88, oom)  \\
\hline
$w=20$ & (0.54, 2.76)  & (1.81, 87.37) & (2.03, 2477.07)& (3.72, oom) \\
\hline
\end{tabular}
\end{table}
\end{eg}

\begin{eg}
\label{ex:quartic}
Consider the sparse convex polynomial optimization
\be  \label{sparse:quartic}
\left\{ \baray{cll}
  \displaystyle \min_{x\in\re^n} &  \displaystyle  \sum _{i=1}^m \Big[ b_i^T\xdt{i} + \xdt{i}^TQ_i\xdt{i} + \big({\xdt{i}^{[2]}}\big)^T D_i \xdt{i}^{[2]}  \Big] \\
 \st  &  1 - c_i^T\xdt{i} - \xdt{i}^TB_i\xdt{i} - \big({\xdt{i}^{[2]}}\big)^T H_i \xdt{i}^{[2]}\ge 0,\ i=1\ddd m.
\earay \right.
\ee
In the above, for $\xdt{i} = \big[ x_{j_1} \ x_{j_2}\ \ldots\ x_{j_{n_i}} \big]^T$, we denote
\[\xdt{i}^{[2]} \,\coloneqq\, \big[ x^2_{j_1} \ x^2_{j_2}\ \ldots\ x^2_{j_{n_i}} \big]^T.\]
We set $m=n$ and each block $\Dt_i$ is given as in (\ref{eq:qcqp_delta_i}).
All $b_i, c_i, Q_i$ and $B_i$ are randomly generated in the same way as in Example~\ref{ex:qcqp}.
Each matrix $D_i = R_i^TR_i$ is randomly generated with the entries
of $R_i$ obeying the uniform distribution on $[0,1]$, and each matrix $H_i$
is generated in the same way.
For these choices, the objective is a sum of SOS-convex polynomials
and the constraining polynomials are SOS-concave.
This can be shown as in \cite[Example~7.1.4]{nie2023moment}.
By Theorem~\ref{tm:sosconvex}, the sparse relaxation (\ref{eq:spa_sos})-(\ref{eq:spa_mom}) is tight for all $k\ge2$.
We report the computational time (in seconds) for solving (\ref{sparse:quartic})
by the sparse relaxation (\ref{eq:spa_sos})-(\ref{eq:spa_mom})
and dense relaxation in \cite{Lasserre2001} with $k=2$.
In Table~\ref{tab:quartic}, we display the computational time by the sparse relaxation (\ref{eq:spa_sos})-(\ref{eq:spa_mom}) on the left,
and the time by the dense relaxation on the right, inside each parenthesis.
When $n$ is $50$ or $100$, the dense relaxations cannot be solved
since the computer is out of memory (oom).

\begin{table}[htb]
\centering
\caption{Computational time (in seconds) for solving (\ref{sparse:quartic})
by the sparse Moment-SOS relaxations (left) and by the dense Moment-SOS relaxation (right).}
\label{tab:quartic}
\begin{tabular}{|l|c|c|c|c|}\hline
& $n=20$ & $n=30$ & $n=50$ & $n=100$  \\
\hline
$w=5$ & (0.10, 17.53)  &  (0.15, $1928.50$) &(0.26, oom)  & (0.67, oom)   \\
\hline
$w=8$ & (0.39, 17.76)  &  (0.63, $1978.11$) &(1.05, oom)  & (2.16, oom)   \\
\hline
$w=10$ & (1.09, 18.82) &  (1.74, $2134.91$) & (3.36, oom) &  (6.96, oom) \\ \hline
\end{tabular}
\end{table}

\end{eg}

We conclude this section with an example such that the sparse Moment-SOS hierarchy (\ref{eq:spa_sos})-(\ref{eq:spa_mom}) is not tight,
while the dense one is tight.

\begin{eg}\label{ex:no_pi}
Consider the optimization problem
\be\label{eq_box_f1+f2}
\left\{
\begin{array}{cl}
\displaystyle \min_{x \in \re^3} & f(x) = \underbrace{x_1^2 + (x_1x_2-1)^2}_{f_1} +
\underbrace{(x_2x_3)^2 + (x_3-1)^2}_{f_2} \\
\st & 1-x_1^2 \ge 0,\ 1-x_2^2 \ge 0,\, 1-x_3^2 \ge 0.
\end{array}\right.
\ee
Here, $\Dt_{1} = \{1, 2\}$ and $\Dt_{2} = \{2, 3\}$.
For fixed $x_2\in[-1,1]$, the minimum value of $f_1(x_1,x_2)$ over $x_1\in[-1,1]$ is
$f_{1, \min} = (x_2^2 + 1)^{-1}$,
which is attained at $x_1=x_2(1+x_2^2)^{-1}$.
Similarly, the minimum value of $f_2(x_2, x_3)$
over $x_3\in[-1,1]$ is $f_{2, \min} = x_2^2(x_2^2 + 1)^{-1},$
which is attained at $x_3 = (1+x_2^2)^{-1}$.
Hence, we have
\[
f_{\min} = f_{1, \min} + f_{2, \min} = 1.
\]
However, Assumption~\ref{eq:f+p>=0} fails to hold. Suppose otherwise it holds,
then there exist $p_1 \in \re[x_1, x_2]$ and $p_2 \in \re[x_2, x_3]$ such that
$p_1  + p_2  + 1 =0$ and
\be\label{eg_suff_cond}
f_1 + p_1 \ge 0 \ \mbox{on} \ [-1,1]^2, \quad f_2 + p_2 \ge 0 \ \mbox{on} \ [-1,1]^2.
\ee
Since $x_2$ is the only joint variable of $\xdt{1}$ and $\xdt{2}$, both $p_1$ and $p_2$ depend only on $x_2$.
So, by the first condition of (\ref{eg_suff_cond}), we have
\be\label{f1min+p1>=0}
f_{1, \min} + p_1(x_2) = (x_2^2+1)^{-1} + p_1(x_2) \ge 0 \quad \forall x_2 \in [-1,1].
\ee
Since $p_2 = -1-p_1$, the second condition of (\ref{eg_suff_cond}) implies
\be\label{f2min-p1>=0}
f_{2, \min} -1- p_1(x_2) = -(x_2^2+1)^{-1}- p_1(x_2)\ge 0 \quad \forall x_2 \in [-1,1].
\ee
Combining (\ref{f1min+p1>=0}) and (\ref{f2min-p1>=0}), we get
\[
p_1(x_2) =  -(x_2^2+1)^{-1}\quad \forall x_2 \in [-1,1].
\]
However, the above cannot hold since $p_1$ is a polynomial.
Therefore, Assumption~\ref{eq:f+p>=0} does not hold for (\ref{eq_box_f1+f2}), and hence the sparse Moment-SOS hierarchy of (\ref{eq:spa_sos})-(\ref{eq:spa_mom}) is not tight.
On the other hand, one can verify that
\[
f(x) - f_{\min} = (x_1x_2 + x_3 - 1)^2 + (x_1 - x_2x_3)^2 ,
\]
which shows that the dense Moment-SOS relaxation is tight.
\end{eg}

\section{Some proofs}
\label{sc:some_proofs}

\subsection{Proof of Theorem~\ref{optmin}}
\label{pf_optmin}

\begin{proof}(Proof of Theorem~\ref{optmin})
Since \reff{decomp_i} holds, we have
\[
M_{\Dt_i}^{(t)}[y^*_{\Dt_i}] = \rho_i [\xdt{i}^*]_{t}[\xdt{i}^*]_{t}^T + W_{\Dt_i}
\]
for a positive scalar $\rho_i >0$ and a moment matrix $W_{\Dt_i} \succeq 0$.
Let
\[
\rho \coloneqq \min_{1 \le i \le m} \rho_i.
\]
Since $(y^*_{\Dt_i})_0=1$, we have $0 < \rho \le \rho_i \le 1$ for every $i\in[m]$, hence
\[
\widehat{W}_{\Dt_i} \, \coloneqq \,
(\rho_i-\rho) [\xdt{i}^*]_{t}[\xdt{i}^*]_{t}^T + W_{\Dt_i} \succeq 0.
\]
Note that $\widehat{W}_{\Dt_i}$ and $W_{\Dt_i}$ are also moment matrices
for some tms's.

For the case $\rho=1$, we have $\widehat{W}_{\Dt_i}=W_{\Dt_i}=0$ for all $i$.
So, $\mathcal{X}_{\Dt_i}$ consists of the single point $\xdt{i}^*$,
and hence $y^*_{\Dt_i}|_{2t} = [\xdt{i}^*]_{2t}$ for all $i\in[m]$.
Then, one can see that
\[
f_k^{smo} = \langle f, y^* \rangle = \sum_{i=1}^m \langle f_i, y^*_{\Dt_i} \rangle
= \sum_{i=1}^m  f_i(\xdt{i}^*) = f(x^*).
\]
Since each $\xdt{i}^* \in K_{\Dt_i}$,
we know $x^*$ is a minimizer of \reff{sparse:pop}.

For the case $0<\rho<1$, define the new tms
\[
\hat{y} \coloneqq (\hat{y}_\af)_{\af \in \U_t} \quad \text{where each} \quad
\hat y_{\af}= \left(x_{\Dt_i}^*\right)^{\af} .
\]
Let $\tilde y \coloneqq  (\tilde y_\af)_{\af\in \U_t}$ be the tms such that
$y^* = \rho \hat y +(1-\rho)\tilde y$, then
\begin{eqnarray*}
	\mathscr{V}_{h_i}^{spa, 2t}[y^*] &=&  \rho \mathscr{V}_{h_i}^{spa, 2t}[\hat{y}] +
	(1-\rho)\mathscr{V}_{h_i}^{spa, 2t}[\tilde{y}], \\
	L_{g_i}^{spa, t}[y^*] &=& \rho L_{g_i}^{spa, t}[\hat{y}]
       +(1-\rho)L_{g_i}^{spa, t}[\tilde{y}], \\
	M_{\Dt_i}^{(t)}[y^*_{\Dt_i}] &=& \rho M_{\Dt_i}^{(t)}[\hat{y}_{\Dt_i}]+
	(1-\rho)M_{\Dt_i}^{(t)}[\tilde{y}_{\Dt_i}].
\end{eqnarray*}
Both $\hat{y}$ and $\tilde{y}$ are feasible for \reff{eq:spa_mom}
with the relaxation order equal to $t$, because
\begin{eqnarray*}
	\mathscr{V}_{h_i}^{spa, 2t}[\tilde{y}] &=& \frac{1}{1-\rho}
       \left(\mathscr{V}_{h_i}^{spa, 2t}[y^*]-
	\rho  \mathscr{V}_{h_i}^{spa, 2t}[\hat{y}] \right) =0, \\
	L_{g_i}^{spa, t}[\tilde{y}] &=& \frac{1}{1-\rho}\left(L_{g_i}^{spa, t}[y^*]-
	\rho L_{g_i}^{spa, t}[\hat{y}] \right) \succeq 0, \\
	M_{\Dt_i}^{(t)}[\tilde{y}_{\Dt_i}] &=& \frac{1}{1-\rho}\left(M_{\Dt_i}^{(t)}[y^*_{\Dt_i}] -
	\rho  M_{\Dt_i}^{(t)}[\hat{y}_{\Dt_i}] \right) \succeq 0.
\end{eqnarray*}
Since $f_t^{smo} =f_k^{smo}$, the truncation $y^*|_{spa,2t}$
is a minimizer for \reff{eq:spa_mom} with $k = t$, so
\[
\langle f, y^* \rangle \le \langle f, \hat{y} \rangle,  \quad
\langle f, y^* \rangle \le \langle f, \tilde{y} \rangle.
\]
On the other hand, it also holds that
\[
f_t^{smo} = \langle f, y^* \rangle = \rho \langle f, \hat{y} \rangle
+ (1 - \rho) \langle f, \tilde{y} \rangle.
\]
Since $0<\rho<1$, we must have
\[
f_t^{smo} = \langle f, y^* \rangle = \langle f, \hat{y} \rangle = \langle f, \tilde{y} \rangle.
\]
Note that $f(x^*) = \langle f, \hat{y} \rangle$, so $f(x^*) = f_t^{smo}$.
Since $f_k^{smo} \le f_{\min}$, we have
\[
f(x^*) = f_t^{smo} \le f_{\min} \le f(x^*).
\]
This shows that $x^*$ is a minimizer of \reff{sparse:pop}.
\end{proof}

\subsection{Proof of Theorem~\ref{thm:spar2full:mom}}
\label{pf_thm:spar2full:mom}

\begin{proof}(Proof of Theorem~\ref{thm:spar2full:mom})
Pick arbitrary $i \ne j$ with $\Dt_{i} \cap \Dt_j \ne \emptyset$,
then \reff{FTDt_ij} holds, by the given assumption. The flat truncation \reff{FTspa}
implies the decomposition \reff{decomp_i}. Note that
\[
\begin{gathered}
	y^*_{\Dt_{ij}} |_{2t} = \lambda_{i,1}[\mathfrak{p}_{ij}(u^{(i,1)})]_{2t} + \cdots + \lambda_{i,r}[\mathfrak{p}_{ij}(u^{(i,r)})]_{2t}, \\
		y^*_{\Dt_{ji}} |_{2t} = \lambda_{j,1}[\mathfrak{p}_{ji}(u^{(j,1)})]_{2t} + \cdots + \lambda_{j,r}[\mathfrak{p}_{ji}(u^{(j,r)})]_{2t}.
\end{gathered}
\]
Note $y^*_{\Dt_{ij}} |_{2t} = y^*_{\Dt_{ji}} |_{2t}$,
since they are common entries of $y^*$.
The condition \reff{FTDt_ij} implies that $y^*_{\Dt_{ij}} |_{2t}$ and $y^*_{\Dt_{ji}} |_{2t}$
have the same unique representing measure, whose support consists of $r$ distinct points.
This is shown in \cite{curto2005truncated,henrion2005detecting,Lau05}.
So, $\mathfrak{p}_{ij}(u^{(i,1)})$, \ldots, $\mathfrak{p}_{ij}(u^{(i,r)})$ are distinct points in  $\mathfrak{p}_{ij}(K_{\Dt_{i}})$, and $\mathfrak{p}_{ji}(u^{(j,1)})$, 
\ldots, $\mathfrak{p}_{ji}(u^{(j,r)})$ 
are distinct points in $\mathfrak{p}_{ji}(K_{\Dt_{j}})$.
Since the representing measure is unique,
\[
\{\mathfrak{p}_{ij}(u^{(i,1)}), \ldots, \mathfrak{p}_{ij}(u^{(i,r)})\} \, =  \,
\{\mathfrak{p}_{ji}(u^{(j,1)}), \ldots, \mathfrak{p}_{ji}(u^{(j,r)})\}.
\]
Up to permutation, we have
\be\label{decomp_eq}
\begin{array}{c}
	\mathfrak{p}_{ij}(u^{(i,1)}) = \mathfrak{p}_{ji}(u^{(j,1)}), \ldots,
	\mathfrak{p}_{ij}(u^{(i,r)}) = \mathfrak{p}_{ji}(u^{(j,r)}), \\
	\lmd_{i,1} = \lmd_{j,1}, \ldots, \lmd_{i,r} = \lmd_{j,r}.
\end{array}
\ee

Next, we show that there exist points $u^{(1)}, \ldots, u^{(r)} \in \re^n$
such that \reff{decomp_spa_general} holds, by induction on $m$.
For the base step (i.e., $m=2$), let $u^{(l)}$ be such that ($l = 1 \ddd r$)
\[
(u^{(l)})_k = (u^{(1,l)})_k \ \text{ for } \ k \in \Dt_1,
\quad (u^{(l)})_k = (u^{(2,l)})_k \ \text{ for } \ k \in \Dt_2.
\]
Since (\ref{decomp_eq}) holds for $i=1$ and $j=2$, the above $u^{(l)}$ is well-defined.
The decomposition \reff{decomp_spa_general} holds for such $u^{(l)}$.
For the inductive step, assume the conclusion holds for $m-1$, and we prove it for $m$.
By the induction, there exist
\[
\bar{u}^{(1)}, \ldots, \bar{u}^{(r)} \in \re^{\bar{n}},
\quad \text{where} \quad  \bar{n} = |\Dt_1 \cup \cdots \cup \Dt_{m-1}|,
\]
such that each $u^{(i,l)}$ is the projection of $\bar{u}^{(l)}$ from $\re^{\bar{n}}$ to $\re^{\Dt_i}$.
Since $\{\Dt_1, \ldots, \Dt_m\}$ is a connected cover of $[n]$ and satisfies the RIP,
we have $\emptyset \ne \Dt_m\cap (\cup_{i=1}^{m-1}\Dt_i) \subseteq \Dt_{i^*}$ for some $i^*\in[m-1]$.
Then, we construct vectors
$
u^{(1)}, \ldots, u^{(r)} \in \re^{n}
$
such that
\[
(u^{(l)})_k = (\bar{u}^{(l)})_k \quad \text{for} \quad k \in \Dt_1 \cup \cdots \cup \Dt_{m-1},
\]
\[
(u^{(l)})_k = (u^{(m,l)})_k \quad \text{for} \quad k \in \Dt_m.
\]
Since (\ref{decomp_eq}) holds for $i=i^*$ and $j=m$, the above $u^{(l)}$ is well-defined.
Therefore, every $u^{(i,l)}$ is the projection of $u^{(l)}$ from $\re^n$ to $\re^{\Dt_i}$.
This means that \reff{decomp_spa_general} holds.

Since $(u^{(j)} )_{\Dt_i} \in K_{\Dt_i}$ for all $i$,
we have $u^{(j)} \in K$ for all $j$.
So every $[u^{(j)}]_{spa,2t}$ is feasible for (\ref{eq:spa_mom}) with the relaxation order equal to $t$,
and every $[u^{(j)}]_{spa,2k}$ is feasible for (\ref{eq:spa_mom}) with the relaxation order equal to $k$.
Therefore,
\[  f_k^{smo} = \lip f,y^*|_{2t} \rip = \lip f,y^*|_{2k} \rip \le
\langle f,[u^{(j)}]_{spa,2k} \rangle = \langle f,[u^{(j)}]_{spa,2t} \rangle, \]
for all $j=1,\ldots,r$.
Moreover, the decomposition (\ref{decomp_spa_general}) implies
\[
f_k^{smo} = \lip f,y^*|_{2t} \rip = \lmd_1 \langle f,[u^{(1)}]_{spa,2t} \rangle + \cdots +
\lmd_r  \langle f,[u^{(r)}]_{spa,2t} \rangle.
\]
Since each $\lmd_i > 0$, the above implies that for all $j=1,\ldots,r$,
\[
f_k^{smo} = \lip f,y^*|_{2t} \rip =
\langle f,[u^{(j)}]_{spa,2t} \rangle = f(u^{(j)})\ge f_{\min}.
\]
Since $f_k^{smo} \le f_{\min}$, all $u^{(1)}\ddd u^{(r)}$
must be minimizers of (\ref{sparse:pop}).
\end{proof}

\subsection{Proof of Theorem~\ref{tm:finite_sosc}}
\label{pf_tm:finite_sosc}

\begin{proof}(Proof of Theorem~\ref{tm:finite_sosc})
(i) Since each $\idea_{\Dt_i}[h_i]+\qm_{\Dt_i}[g_i]$ is archimedean,
the feasible sets $K_{\Dt_i}$ and $K$ are all compact.
So, \reff{sparse:pop} achieves its minimum value, say, at a minimizer $u \in K$.
Assumption~\ref{eq:f+p>=0} implies
\[
(f_1+p_1) + \cdots + (f_m+p_m) = f - f_{\min}.
\]
Since $f(u) = f_{\min}$ and each $f_i+p_i \ge 0$ on $K_{\Dt_i}$,
the minimum value of \reff{eq:split_pop} is $0$.
By \cite[Theorem~1.1]{Nie14},
if the LICQC, SCC, and SOSC hold at every minimizer of \reff{eq:split_pop},
then there is a relaxation order $N_i$ such that for all $\eps>0$,
\be
\label{eq:fi+pi=phi+si}
f_i+p_i- (-\eps) \in
\idea_{\Dt_i}[h_i]_{2N_i} + \qm_{\Dt_i}[g_i]_{2N_i}.
\ee
Let $N \coloneqq \max\{ N_1\ddd N_m\}$.
The above then implies
\[
f+p_1+\cdots+p_m + m\epsilon
\in \ideal{h}_{spa,2N}+\qmod{g}_{spa,2N}.
\]
By Assumption~\ref{eq:f+p>=0}, we have $p_1+\dots+p_m+f_{\min} = 0$,
so for all $k \ge N$
\begin{eqnarray*}
	f- (f_{\min} - m\epsilon) &=&  f+\sum_{i=1}^m p_i + m\epsilon -
	\Big(\sum_{i=1}^m p_i+f_{\min} \Big) \\
	& \in & \ideal{h}_{spa,2k}+\qmod{g}_{spa,2k}.
\end{eqnarray*}
Thus, for arbitrary $\eps > 0$, we have
$
f_{k}^{spa} \ge f_{\min} - m\eps.
$
On the other hand, $f_{k}^{spa} \le f_{\min}$,
so this forces $f_{k}^{spa} = f_{\min}$.

Moreover, if each $\idea_{\Dt_i}[h_i]$ is real radical,
then
\be
\label{eq:fi+pi=phi+si_radical}
f_i+p_i - 0 \in \idea_{\Dt_i}[h_i] + \qm_{\Dt_i}[g_i].
\ee
This follows from \cite[Theorem~9.5.3]{marshall2008positive}
and \cite[Theorem~3.1]{Nie14}. So, we get
\[
f+p_1+\ldots+p_m \in \ideal{h}_{spa}+\qmod{g}_{spa}.\]
By Assumption~\ref{eq:f+p>=0}, $p_1+\cdots+p_m+f_{\min}=0$, so
\[f-f_{\min} = f+\sum_{i=1}^m p_i - \Big(\sum_{i=1}^m p_i+f_{\min} \Big)
\in \ideal{h}_{spa}+\qmod{g}_{spa}.
\]

\smallskip \noindent
(ii) For each degree $k\ge k_0$, the $k$th order SOS relaxation for \reff{eq:split_pop} is
\be \label{eq:sos}
\left\{ \baray{cll}
\max &  \gamma  \\
\st  &  f_i+p_i - \gamma  \in  \idea_{\Dt_{i}}[h_i]_{2k} + \qm_{\Dt_i}[g_i]_{2k}.
\earay \right.
\ee
Its dual optimization is the moment relaxation
\be\label{eq:mom}
\left\{ \baray{cll}
\min &  \langle f_i+p_i, y_{\Dt_{i}} \rangle \\
\st   &  \mathscr{V}_{h_i}^{\Dt_i, 2k}[y_{\Dt_{i}}] = 0,
   L_{g_i}^{\Dt_i, k}[y_{\Dt_{i}}] \succeq 0, \\
&  (y_{\Dt_{i}})_0=1, \ M_{\Dt_i}^{(k)}[y_{\Dt_{i}}]\succeq 0,\\
& y_{\Dt_i} \in \re^{\N^{\Dt_i}_{2k}}.
\earay \right.
\ee
The minimum value of \reff{eq:split_pop} is $0$.
Since each $\idea_{\Dt_i}[h_i]+\qm_{\Dt_i}[g_i]$ is archimedean and the LICQC, SCC, SOSC hold at each optimizer of \reff{eq:split_pop},
the Moment-SOS hierarchy of \reff{eq:sos}-\reff{eq:mom} has finite convergence.
This is shown in \cite{Nie14}.
So, there exists $N_0$ such that for all $k \ge N_0$ and all $\eps > 0$,
\[
f_i + p_i -(-\eps) \in \idea_{\Dt_{i}}[h_i]_{2k} + \qm_{\Dt_{i}}[g_i]_{2k}.
\]
Since $y^*$ is a minimizer of \reff{eq:spa_mom}, $y^*_{\Dt_i}$ is feasible for \reff{eq:mom}. Hence,
\[
\langle f_i + p_i , y^*_{\Dt_i} \rangle + \eps = \langle f_i + p_i + \eps, y^*_{\Dt_i} \rangle \ge 0.
\]
Since $\eps > 0$ can be arbitrary, we get $\langle f_i + p_i , y^*_{\Dt_i} \rangle \ge 0$.
By item (i), it holds
\begin{eqnarray*}
	0 &=& f^{smo}_k - f_{\min} = \langle f, y^* \rangle - f_{\min} =
	\langle f - f_{\min}, y^* \rangle \\
	&=& \Big\langle \sum_{i=1}^{m} f_i + p_i , y^* \Big\rangle =
	\sum_{i=1}^{m} \langle f_i+p_i, y^*_{\Dt_i} \rangle.
\end{eqnarray*}
In the above, we have used $-f_{\min} = p_1 + \cdots + p_m$ by Assumption~\ref{eq:f+p>=0}.
So, $\langle f_i + p_i , y^*_{\Dt_i} \rangle = 0$ for every $i$.
This means that $y^*_{\Dt_i}$ is a minimizer for \reff{eq:mom}.
Therefore, the conclusion follows from
Theorem~3.3 of \cite{huang2023finite} and Theorem~\ref{optmin}.

\smallskip \noindent
(iii) Since LICQC, SCC and SOSC hold at each minimizer of \reff{eq:split_pop},
the set $S_i$ of all minimizers of \reff{eq:split_pop} is finite.
Note that $\mathcal{X}_{\Dt_i} \subseteq S_i$, 
so $\mathfrak{p}_{ij}(\mathcal{X}_{\Dt_i})$ is also a finite set.
Pick arbitrary $i \ne j$ with $\Dt_{i} \cap \Dt_j \ne \emptyset$,
then $\mathfrak{p}_{ij}(\mathcal{X}_{\Dt_i})$
is a set of $r_i$ distinct points, by the given assumption.
We know that $y^*_{\Dt_{ij}}$ has a representing measure whose support is $\mathfrak{p}_{ij}(\mathcal{X}_{\Dt_i})$. When $t$ is big enough, we must have
\[
\rank \, M_{\Dt_{ij} }^{(t)}[y^*_{\Dt_{ij}}] =
\rank \, M_{\Dt_{ij}}^{(t-1)}[y^*_{\Dt_{ij}}].
\]
This is because $\rank \, M_{\Dt_{ij} }^{(t)}[y^*_{\Dt_{ij}}]$
is uniformly bounded above by the cardinality of
$\mathfrak{p}_{ij}(\mathcal{X}_{\Dt_i})$.
So, the truncation $y^*_{\Dt_{ij}}|_{2t}$ is flat and
it has a  unique representing measure, say, $\mu_{ij}$.
Since $y^*_{\Dt_{ij}}|_{2t} = y^*_{\Dt_{ji}}|_{2t}$, we have $\mu_{ij} = \mu_{ji}$.
Hence, $\mathfrak{p}_{ij}(\mathcal{X}_{\Dt_i}) =  \mathfrak{p}_{ji}(\mathcal{X}_{\Dt_j})$.
Since $\mathfrak{p}_{ij}(\mathcal{X}_{\Dt_i})$ consists of $r_i$ distinct points,
we must have $r_i = r_j$. Therefore, $r_1 = \cdots = r_m = r$, since
$\{\Dt_1, \ldots, \Dt_m\}$ is a connected cover of $[n]$.
Hence, \reff{FTyij} holds.
The remaining conclusions follow from Theorem~\ref{thm:spar2full:mom}.
\end{proof}

\subsection{Proof of Theorem~\ref{tm:finite_variety}}
\label{pf_tm:finite_variety}

\begin{proof}(Proof of Theorem~\ref{tm:finite_variety})
(i) Since each $V_{\re}(h_i)$ is finite, there exists $N_i$ such that for all $\eps > 0$,
\[
f_i + p_i + \eps \in \idea_{\Dt_i}[h_i]_{2N_i} + \qm_{\Dt_i}[g_i]_{2N_i}.
\]
This is shown in the proof of Theorem~1.1 of \cite{Nie13Finite}.
Let $N \coloneqq \max\{N_1,\ldots, N_m\}.$
Then, for all $\eps>0$, we have
\begin{eqnarray*}
	f-(f_{\min}-m\eps) &  =  &  \sum_{i=1}^m \big[f_i + p_i + \eps\big] -
      \Big(\sum_{i=1}^m p_i +f_{\min} \Big) \\
	& \in   &  \ideal{h}_{spa,2N} + \qmod{g}_{spa,2N}.
\end{eqnarray*}
So, $\gamma = f_{\min}-m\eps$ is feasible for \reff{eq:spa_sos} with $k=N$.
Hence, for all $\eps>0$,
\[
f_{\min}-m\eps \le f_{N}^{spa} \le f_{\min}.
\]
This forces $f_{k}^{spa} = f_{\min}$ for all $k\ge N$.

\smallskip
Furthermore, if each $\idea_{\Dt_i}[h_i]$ is real radical, then
\[
f_i + p_i  \in \idea_{\Dt_i}[h_i]_{2N_i} + \qm_{\Dt_i}[g_i]_{2N_i},
\]
when $N_i$ is big enough.
This is implied by Proposition~5.6.4 of \cite{nie2023moment}.
So, \reff{fminachieved:finite_var} follows.

\smallskip \noindent
(ii) Consider the relaxations \reff{eq:sos}-\reff{eq:mom}.
Since $y^*$ is a minimizer of \reff{eq:spa_mom},
$y^*_{\Dt_i}$ is a minimizer of \reff{eq:mom}.
This can be similarly shown as for Theorem~\ref{tm:finite_sosc}(ii).
Since $V_{\re}(h_i)$ is finite, \reff{FTspa_finite_var}
must be satisfied, by Theorem~5.6.1 of \cite{nie2023moment}.

\smallskip \noindent
(iii) Note that $\mathcal{X}_{\Dt_i} \subseteq V_{\re}(h_i)$,
so $\mathfrak{p}_{ij}(\mathcal{X}_{\Dt_i})$ is also a finite set.
The remaining proof is the same as for Theorem~\ref{tm:finite_sosc}(iii).
\end{proof}

\subsection{Proof of Theorem~\ref{tm:ripfiniteass}}
\label{sc:proof_tm46}

We first prove Theorem~\ref{tm:ripfiniteass} for the case $m=2$.

\begin{lem}\label{lm:block2}
Consider (\ref{sparse:pop}) with $m=2$. Assume $f_{\min} > -\infty$,
its feasible set $K \ne \emptyset$
and the following two optimization problems are bounded below:
\be  \label{blocko}
\left\{ \baray{cll}
\displaystyle \min &     f_1(\xdt{1}) \\
  \st  &  h_1( \xdt{1} ) = 0, \\
      &  g_1( \xdt{1} ) \ge 0, \
\earay \right. \quad
\left\{ \baray{cll}
\displaystyle \min &     f_2(\xdt{2}) \\
  \st  &  h_2( \xdt{2} ) = 0 , \\
      &  g_2( \xdt{2} ) \ge 0. \
\earay \right.
\ee
If the projections $\mathfrak{p}_{12}(K_{\Delta_1})$ and $\mathfrak{p}_{21}(K_{\Delta_2})$
are both finite sets, then there exists $p\in \re[x_{\Delta_{12}}]$ such that
\be \label{ass(4.1):m=2}
f_1+p \ge 0 \,\  \mathrm{ on } \,\  K_{\Delta_1}, \quad
f_2-p-f_{\min} \ge 0 \,\  \mathrm{ on } \,\  K_{\Delta_2},
\ee
the minimal values of $f_1+p$ on $K_{\Delta_1}$ and
$f_2-p-f_{\min}$ on $K_{\Delta_2}$ are both zeros.
\end{lem}
\begin{proof}
By the assumption, we can write that
\[
\mathfrak{p}_{12}(K_{\Delta_1}) \cup \mathfrak{p}_{21}(K_{\Delta_2}) =
\{u^{(1)},\ldots,u^{(D)}\}\subseteq \re^{{\Delta_{12}}}.
\]
Let $\varphi_{1}, \ldots, \varphi_{D} \in \re[x_{\Delta_{12}}]$
be interpolating polynomials such that
$\varphi_i(u^{(j)}) = 0$ for $i \ne j$ and
$\varphi_i(u^{(j)}) = 1$ for $i = j$.
Denote the sets
\be \label{sets:U0U1U2}
U_0\coloneqq\mathfrak{p}_{12}(K_{\Delta_1}) \cap \mathfrak{p}_{21}(K_{\Delta_2}),\enspace U_1\coloneqq\mathfrak{p}_{12}(K_{\Delta_1})\backslash U_0,\enspace U_2\coloneqq\mathfrak{p}_{21}(K_{\Delta_2})\backslash U_0.
\ee
Define the optimal value functions:
\begin{align}
\label{firstblock}
\left\{ \baray{rcl}
F_1(u) \coloneqq & \displaystyle \min &     f_1( \xdt{1} ) \\
& \st & \xdt{1} \in K_{\Dt_1}, \,  x_{\Delta_{12}}=u ,
\earay \right.
\\
\label{secondblock}
\left\{ \baray{rcl}
F_2(u) \coloneqq & \displaystyle \min &     f_2( \xdt{2} ) \\
& \st & \xdt{2} \in K_{\Dt_2}, \, x_{\Delta_{12}}=u .
\earay \right.
\end{align}
For each $u\in U_0\cup U_1$, the problem (\ref{firstblock})
is feasible and its optimal value is finite.
Similarly, for each $u\in U_0\cup U_2$, the problem (\ref{secondblock})
is feasible and its optimal value is finite.
For all $u\in U_0$, $F_1(u) + F_2(u) $ equals the minimum value of
\[
\left\{ \baray{cll}
 \displaystyle \min &    f_1(\xdt{1}) + f_2( \xdt{2} ) \\
 \st & \xdt{1} \in K_{\Dt_1}, \,  \xdt{2} \in K_{\Dt_2}, \,  x_{\Delta_{12}}=u.
\earay \right.
\]
Since (\ref{sparse:pop}) is feasible and $U_0$
is given as in \reff{sets:U0U1U2},  we have
\be\label{eq:F1+F2}
\min_{u\in U_0} [F_1(u)+F_2(u)] \,\, = \,\, f_{\min}.
\ee
For $i = 1, \ldots, D$, denote the values
\[
v_i\coloneqq\left\{
\begin{array}{rl}
-F_1 (u^{(i)})  &  \mathrm{~if~} u^{(i)}\in U_0\cup U_1, \\
F_2(u^{(i)}) -f_{\min}   & \mathrm{~if~} u^{(i)}\in  U_2 .
\end{array}
\right.
\]
Let $p \coloneqq  v_1\varphi_1 + \cdots + v_D\varphi_D$,
then $p\in \re[x_{\Delta_{12}}]$ and
\be\label{eq:pu}
p(u)=\left\{
\begin{array}{rl}
-F_1 (u)  &  \mathrm{~if~} u\in U_0\cup U_1, \\
F_2(u)  -f_{\min}    & \mathrm{~if~} u\in  U_2.
\end{array}
\right.
\ee
Observe the relations
\[
\min_{x_{\Delta_1}\in K_{\Delta_1}} \big[ f_1(x_{\Delta_1})+p(x_{\Delta_{12}}) \big] =
\min_{u\in \mathfrak{p}_{12}(K_{\Delta_1})} \big[ F_1(u)+p(u) \big],
\]
\[
\min_{x_{\Delta_2}\in K_{\Delta_2}} \big[ f_2(x_{\Delta_2})-p(x_{\Delta_{12}})-f_{\min} \big] =
\min_{u\in \mathfrak{p}_{21}(K_{\Delta_2})} \big[ F_2(u)-p(u)-f_{\min} \big].
\]
Since $\mathfrak{p}_{12}(K_{\Delta_1}) = U_0\cup U_1$, \eqref{eq:pu} implies
\[
\min_{u\in \mathfrak{p}_{12}(K_{\Delta_1})} [F_1(u)+p(u)] \,\, = \,\,  0.
\]
Since $\mathfrak{p}_{21}(K_{\Delta_2}) = U_0\cup U_2$, it holds
\[
\min_{u\in \mathfrak{p}_{21}(K_{\Delta_2})} [F_2(u)-p(u)-f_{\min}] \, = \,
\min_{u\in U_0\cup U_2} [F_2(u)-p(u)-f_{\min}].
\]
Also note that
\[
F_2(u) - p(u) = f_{\min} \,\, \mbox{if} \, \, u \in U_2, \quad
F_2(u) - p(u) = F_1(u) + F_2(u)  \,\, \mbox{if} \,\, u\in U_0.
\]
So, $ \min\limits_{u\in U_2} [F_2(u)-p(u)-f_{\min}] = 0$, and (\ref{eq:F1+F2}) implies
\[
\min_{u\in U_0} [F_2(u)-p(u)-f_{\min}]  \, = \,
\min_{u\in U_0} [F_1(u)+F_2(u)-f_{\min}] = 0.
\]
This completes the proof.
\end{proof}

In the following, we prove Theorem~\ref{tm:ripfiniteass} by induction.

\begin{proof}[Proof of Theorem~\ref{tm:ripfiniteass}.]
The conclusion holds when $m=2$, by Lemma~\ref{lm:block2}.
Suppose it holds for $m=k$. We prove it also holds for $m=k+1$.
The problem~\eqref{sparse:pop} can be viewed to have two sparsity blocks:
$\widehat{\Delta}_{m-1}$ and $\Dt_{m}$ (see \reff{hatDt:t-1} for the notation).
Consider the following two optimization problems:
\be  \label{blockom}
\left\{ \baray{cl}
  \min &  f_1(\xdt{1}) + \cdots + f_{m-1}(\xdt{m-1}) \\
 \st & \xdt{i} \in K_{\Dt_i}, \,\, i =1, \ldots,  m-1,
\earay \right.
\quad
\left\{ \baray{cl}
\min & f_m(\xdt{m})  \\
 \st & \xdt{m} \in K_{\Dt_m}.
\earay \right.
\ee
The feasible set for the above left optimization problem is
$K_{\widehat{\Dt}_{m-1}}$ (see \reff{KhatDt:t-1} for the notation).
By the assumption, both $K_{\widehat{\Dt}_{m-1}}$ and $K_{\Delta_m}$
are nonempty and compact; the optimal values of both optimization problems
in \reff{blockom} are finite.
By Lemma~\ref{lm:block2}, there exists $p\in \re[x_{\widehat{\Dt}_{m-1}}]\cap\re[x_{\Dt_m}]$ such that
\be\label{eq:m2blockm}\left\{
\begin{array}{ll}
 f_1+\cdots+f_{m-1} + p \ge 0\quad \mbox{on} \quad K_{\widehat{\Dt}_{m-1}},\\
 f_m - p -f_{\min}\ge 0\quad \mbox{on} \quad K_{\Delta_m}.
\end{array}\right.
\ee
Furthermore, the minimum values of $f_1+\cdots+f_{m-1} + p$ on $K_{\widehat{\Dt}_{m-1}}$
and $f_m - p -f_{\min}$ on $K_{\Delta_m}$ are both equal to $0$.
The RIP ensures that there exists $i^*\in [m-1]$ such that
$\widehat{\Dt}_{m-1}\cap \Dt_m\subseteq \Delta_{i^*}$.
For $i = 1, \ldots, m-1$, let
\[
\tilde f_i  \coloneqq
\left\{\begin{array}{ll}
f_i,  & \enspace   \mathrm{~if~} i\neq i^*,\\
f_i+p, & \enspace \mathrm{~if~} i=i^*.
\end{array}
\right.
\]
Then $\tilde f_i\in \re[x_{\Delta_i}]$ for all $i = 1, \ldots, m-1$
and the minimum value of $\tilde f_1+\cdots+\tilde f_{m-1} $
on $K_{\widehat{\Dt}_{m-1}}$ is  zero.
By the induction hypothesis, there exist polynomials $\tilde{q}_i  \in \re[\xdt{i}]$ such that
\be  \label{assumpequivqm-1}
\left\{
\begin{gathered}
\tilde q_1  + \cdots + \tilde q_{m-1}   =0 , \\
\tilde f_i + \tilde q_i \ge 0~ \text{on} ~ K_{\Dt_i},\,
i=1\ddd m-1 .
\end{gathered}\right.
\ee
Let $q_m\coloneqq-p-f_{\min}$,
and for each $i = 1, \ldots, m-1$, let
$$
 q_i \coloneqq \left\{\begin{array}{ll}
\tilde q_i,  & \enspace   \mathrm{~if~} i\neq i^*,\\
\tilde q_i+p, & \enspace \mathrm{~if~} i=i^*.
\end{array}\right.
$$
Then, we have $q_i \in \re[x_{\Delta_i}]$ for every $i\in [m]$ and
\[\left\{\begin{gathered}
q_1+\cdots+q_m=\tilde q_1+\cdots+\tilde q_{m-1}-f_{\min},\\
f_m+q_m=f_m-p-f_{\min},\\
f_i+q_i =\tilde f_i+\tilde q_i,\enspace   i = 1, \ldots, m-1.
\end{gathered}\right. \]
Therefore, \eqref{eq:m2blockm} and~\eqref{assumpequivqm-1} imply that
\be  \label{assumpequivqm}
\left\{
\begin{gathered}
 q_1  + \cdots +  q_{m} +f_{\min}   =0 , \\
 f_i +  q_i \ge 0~ \text{on} ~ K_{\Dt_i}, \, i=1\ddd m .
\end{gathered}\right.
\ee
This completes the proof.
\end{proof}

\subsection{Proof of Theorem~\ref{thm:schm:finitevariety}}
\label{pf_thm:schm:finitevariety}

\begin{proof}(Proof of Theorem~\ref{thm:schm:finitevariety})
(i) As we have shown in Theorem~\ref{tm:finite_sosc}(i),
the minimum value of \reff{eq:split_pop} is $0$.
Since $K_{\Dt_i}$ is finite, there exists a degree $N_i$ such that for all $\eps > 0$,
\[
f_i + p_i+ \epsilon \in \idea_{\Dt_i}[h_i]_{2N_i}
+ \mbox{Pre}_{\Dt_i}[g_i]_{2N_i}.
\]
This is shown in the proof of Theorem 4.1 of \cite{Nie13Finite}.
Note that
\begin{align*}
	\sum_{i=1}^m\left(f_i + p_i + \epsilon \right)
	= f + m\epsilon - f_{\min}  +\Big (\sum_{i=1}^mp_i + f_{\min} \Big ).
\end{align*}
Let $N \coloneqq \max\left\{N_1\ddd N_m \right\}.$
Then, for all $\eps>0$, we have
\begin{eqnarray*}
	f - (f_{\min} - m \eps) &=& \sum_{i=1}^m\left(f_i + p_i + \epsilon \right) -
	\Big ( \sum_{i=1}^mp_i + f_{\min} \Big ) \\
	&\in& \ideal{h}_{spa, 2N} + \pre{g}_{spa, 2N}.
\end{eqnarray*}
This implies that $\gm = f_{\min} - m\eps$ is feasible for
\reff{eq:sch_sos} with the order $k=N$. Hence, for all $\eps>0$,
\[
f_{\min}-m\eps \le f_{N}^{smg} \le f_{\min}.
\]
This forces $f_{N}^{smg} = f_{\min}$, so
$f_{k}^{smg} = f_{\min}$ for all $k\ge N$.

\smallskip \noindent
(ii) Suppose $y^*$ is a minimizer of \reff{eq:sch_mom}.
In the same way as for Theorem~\ref{tm:finite_sosc}(ii),
we can show that each $y^*_{\Dt_{i}}$ is a minimizer of the moment relaxation
\be\label{eq:smgKi}
\left\{ \baray{cll}
\min &  \langle f_i + p_i, y_{\Dt_{i}} \rangle \\
\st   &  \mathscr{V}_{h_i}^{\Dt_{i}, 2k}[y_{\Dt_{i}}] = 0 \, (i \in [m]),  \\
&  L_{g_{i,J}}^{\Dt_{i}, k}[y_{\Dt_{i}}] \succeq 0 \,
(i \in [m], \, J \subseteq [s_i]), \\
&  (y_{\Dt_{i}})_0=1, \ y_{\Dt_{i}} \in \re^{\N^{\Dt_{i}}_{2k}}.
\earay \right.
\ee
Since each $K_{\Dt_i}$ is finite, by \cite[Theorem~5.6.7]{nie2023moment},
we know $y^*_{\Dt_i}$ satisfies the flat truncation \reff{FT_smg} when $k$ is large enough.
The remaining part follows from Theorem~\ref{optmin}.

\smallskip \noindent
(iii) Note that $\mathcal{X}_{\Dt_i} \subseteq K_{\Dt_i}$,
so $\mathfrak{p}_{ij}(\mathcal{X}_{\Dt_i})$ is also a finite set.
Then, by the same proof as for Theorem~\ref{tm:finite_sosc}(iii),
we can get \reff{FTyDt_ij} and \reff{dec_5.7}.
\end{proof}

%
%

\section{Conclusions and Discussions}
\label{sc:conclusion}

This paper studies the sparse Moment-SOS hierarchy of relaxations
\reff{eq:spa_sos}-\reff{eq:spa_mom}
for solving sparse polynomial optimization problem \reff{sparse:pop}.
We show that this sparse Moment-SOS hierarchy is tight if and only if
\reff{eq:p+fmin_in_idl} or \reff{eq:sep_pi_eps} holds, i.e.,
the objective can be equivalently written as a sum of sparse nonnegative polynomials,
each of which belongs to the sum of the ideal and quadratic module
generated by the corresponding constraints.
Under Assumption~\ref{eq:f+p>=0}, we give some sufficient conditions
for this sparse hierarchy to be tight: optimality conditions for
\reff{eq:split_pop} or finiteness of individual constraining sets.
We also prove some conditions for Assumption~\ref{eq:f+p>=0} to hold.
In particular, we show that the sparse Moment-SOS hierarchy is tight
under some convexity assumptions.

Here are some interesting questions for future work.
\bit

\item
When the RIP holds, if $K$ is a finite set but each $K_{\Dt_i}$ may not,
is the Schm\"{u}dgen-type sparse hierarchy
of \reff{eq:sch_sos}-\reff{eq:sch_mom} always tight?

\item
Does Assumption~\ref{eq:f+p>=0} hold when
$f_i(\xdt{i})$, $g_i(\xdt{i})$, $h_i(\xdt{i})$ are generic polynomials?

\item
When the RIP holds, does the sparse Positivstellensatz hold?
That is, when the RIP holds and $K=\emptyset$,
do there exist $\sig \in \pre{g}_{spa}$ and
$\phi \in \ideal{h}_{spa}$ such that (\ref{positiv}) holds?

\eit

\subsection*{Acknowledgements}

This project was begun at a SQuaRE at the American Institute of Mathematics (AIM).
The authors would like to thank AIM for providing
a supportive and mathematically rich environment.
Jiawang Nie and Linghao Zhang are partially supported by
the National Science Foundation grant DMS-2110780.
Zheng Qu is partially supported by the Research Center for Intelligent Operations Research
at Department of Applied Mathematics, the Hong Kong Polytechnic University.
Xindong Tang is partially supported by the Hong Kong Research Grants Council HKBU-15303423.

\end{document}